\RequirePackage{fix-cm}
\documentclass[final,12pt]{elsarticlehack}
\usepackage{amsmath, amsfonts, amsthm, amssymb}
\usepackage{placeins}
\usepackage{pdfpages}

\usepackage[ruled,vlined,linesnumbered,norelsize]{algorithm2e}
\usepackage{comment}
\usepackage{placeins}
\usepackage{subfigure}

\usepackage{nicematrix} %
\usepackage{tikz}
\usetikzlibrary{fit}

\usepackage[hypertexnames=false,colorlinks=true,breaklinks=true,bookmarks=true,urlcolor=blue,citecolor=blue,linkcolor=blue,bookmarksopen=false,draft=false]{hyperref}

\theoremstyle{plain}
\newtheorem{theorem}{Theorem}

\makeatletter
\renewcommand*{\top}{%
  {\mathpalette\@transpose{}}%
}
\newcommand*{\@transpose}[2]{%
  \raisebox{\depth}{$\m@th#1\scriptscriptstyle\mathsf{T}$}%
}
\makeatother

\allowdisplaybreaks[1]

\begin{document}
\begin{frontmatter}
\title{On a geometric graph-covering problem related to optimal safety-landing-site location\tnoteref{label0}}
\tnotetext[label0]{A short preliminary version appeared in the proceedings of ISCO 2024; see \cite{Covering_ISCO}.}

\author[label1]{Claudia D'Ambrosio}
\author[label2]{Marcia Fampa}
\author[label3]{Jon Lee}
\author[label2]{Felipe Sinnecker}

\affiliation[label1]{organization={Ecole Polytechnique},%Department and Organization
            % city={},
            country={France}}
            
\affiliation[label2]{organization={Universidade Federal do Rio de Janeiro},%Department and Organization
            % city={},
            country={Brazil}}

\affiliation[label3]{organization={University of Michigan},%Department and Organization
            city={Ann Arbor},
            state={MI},
            country={USA}}

% \affiliation[label4]{organization={Universidade Federal do Rio de Janeiro},%Department and Organization
%             % addressline={}, 
%             % city={},
%             % postcode={}, 
%             % state={},
%             country={Brazil}}
            
\begin{abstract}
% We develop a set-cover based integer-programming approach to 
% an optimal safety-landing-site location arising in the design of urban air-transportation networks. We link these set-cover problems to solvable cases that have been studied. We introduce \emph{strong fixing} which we found to be very effective in reducing the size of the integer programs. 
% {\color{blue} add about \ref{CP2}.}
We propose integer-programming formulations for an optimal safety-landing site (SLS) location problem that arises in the design of urban air-transportation networks. We first develop a set-cover based approach for the case where the candidate location set is finite and composed of points, and we link the problems to solvable cases that have been studied. We then use a mixed-integer second-order cone program to model the situation where the locations of SLSs are restricted to convex sets only. Finally, we introduce \emph{strong fixing}, which we found to be  very effective in reducing the size of integer programs.
\end{abstract}


\end{frontmatter}

\section*{Introduction}
In the last few years, different actors all around the world have been pushing for the development of Urban Air Mobility (UAM). The idea is to integrate into the current transportation system, new ways to move people and merchandise. In particular, drones are already a reality, and they have the potential to be highly exploited for last-mile deliveries (for example, Amazon, UPS, DHL, and FedEx, just to mention a few). Concerning passenger transportation, several companies are competing to produce the first commercial electric Vertical Take-Off and Landing (eVTOLs) vehicles, which will be used to move passengers between skyports of sprawling cities. Several aspects have to be taken into account for this kind of service, the most important one being safety. Air-traffic management (ATM) provides and adapts flight planning to guarantee a proper separation of the trajectories of the flights; see e.g., \cite{PDDH23,CDLP21,PD22}. In the case of UAM, some infrastructure has to be built to provide safe landing locations in case of failure or damage of drones/eVTOLs. These locations are called ``Safety Landing Sites'' (SLSs) and should be organized  to cover the trajectory of eVTOLs for emergency landings at any position along flight paths.

In what follows, we study the optimal placement of SLSs in the air-transpor\-tation network. We aim at installing the minimum-cost set of SLSs, such that all the drones/eVTOLs trajectories are covered. We show an example of in Fig. \ref{fig:example_SLScover}. It represents an aerial 2D view of a part of a city, where the rectangles are the roofs of existing buildings. The black crosses represent potential sites for SLSs. The two red segments are the trajectory of the flights in this portion of the space. The trajectory is fully covered thanks to the installation of 3 SLS over 5 potential sites, namely the ones corresponding to the center of the green circles. The latter represents the points in space that are at a distance that is smaller than the safety distance for an emergency landing. Note that every point along the trajectory is inside at least one circle, thus the trajectory is fully covered.

\begin{figure}
\centering
\includegraphics[width=0.25\textwidth]{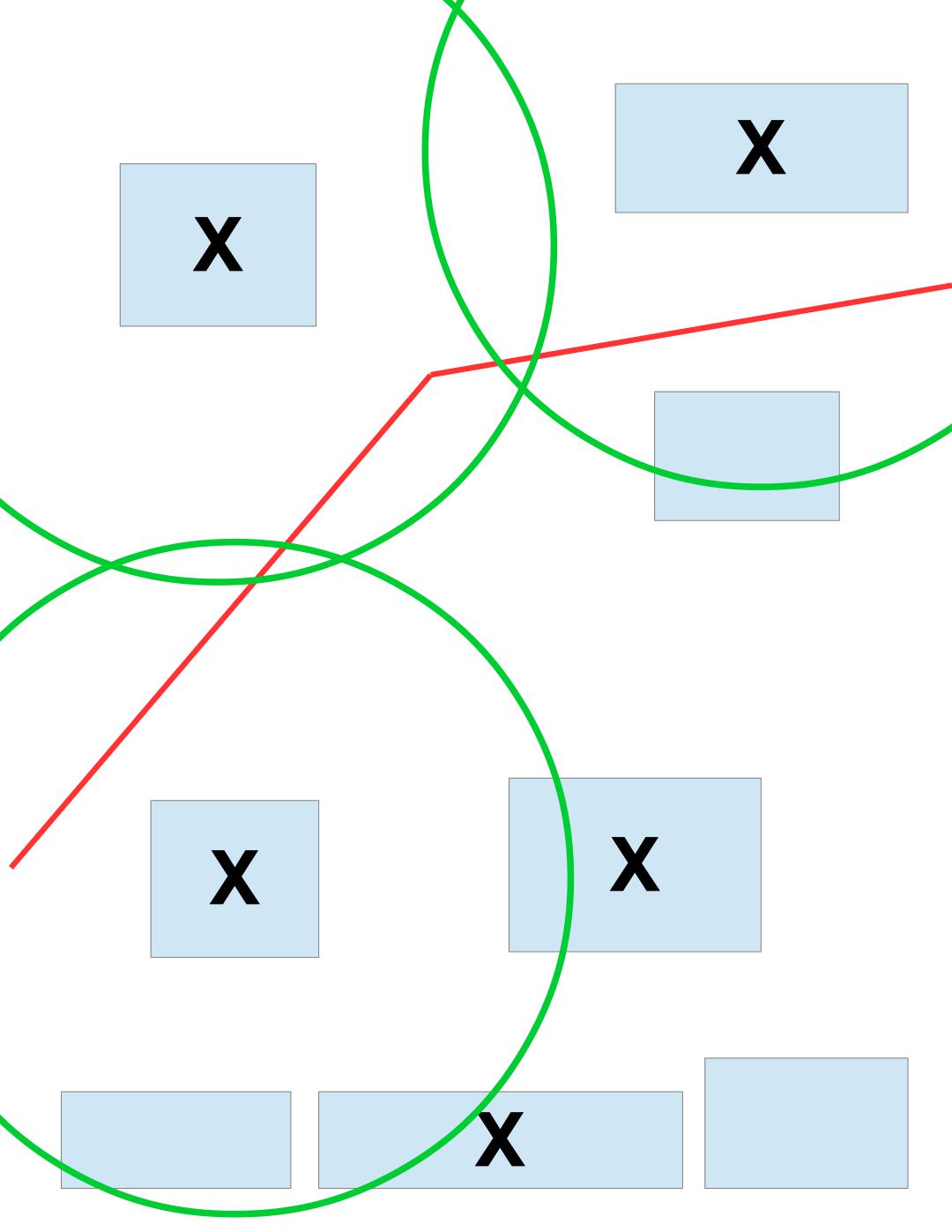}
\caption{Example of full covering provided by installing 3 SLSs}\label{fig:example_SLScover}
\end{figure}

The problem of finding an optimal placement of SLSs has not received a lot of attention. In fact, there appears to be no published work on the variants that we are considering. In the literature, we can find the master's thesis of Xu \cite{xu:hal-03286640}, where he studied the problem of SLS placement coupled with the routing problem of drones/eVTOLs for each origin-destination pair. The set of potential SLSs location is finite and they have a constraint on the budget of SLSs to be installed. The potential SLSs are assumed to fully cover a subset of the arcs of the considered network, i.e., a partial covering is not allowed. In \cite{pelegrín2023continuous}, Pelegr\'{\i}n and Xu consider a variant of the problem which can be formulated as a continuous covering problem. In particular, they model the problem as a set-covering/location problem so that both candidate locations and demand points are continuous on a network. The demand points are any point along the potential path of any eVTOL.
In this work, we address two versions of the problem where  the demand points are continuous on the network and the candidate-location set is  not restricted to be in the network. In the first version, the  candidate-location set is finite and composed of points; see e.g., Fig.~\ref{fig:example_SLScover}. We then address the case where  the locations of SLSs are restricted to convex sets only. 

More attention has been accorded to the optimal placement of ``vertiports''  (the short term for ``vertical ports''); see e.g., \cite{Villa2020,SMA2021}. However, their location depends on the estimated service demand and, based on the decisions made on the vertiport location, the UAM network is identified. In contrast, we suppose that these decisions were already made. In fact, 
despite the scarcity of literature on the topic, the main actors in the UAM field assert that pre-identified emergency landing sites are necessary to guarantee safety in UAM; see  e.g., \cite{NASA2021,FAA_NASA_2023}.

\medskip
\noindent {\bf Organization and contributions}.
In \S\ref{sec:covergiven}, we describe our first new mathematical model for the optimal safety-landing-site (SLS) location problem, a generally NP-hard minimum-weight set-covering problem. 
In \S\ref{sec:matrices}, we see what kinds of set-covering matrices can arise from our setting, linking to the literature on ideal matrices. 
In \S\ref{sec:solveable}, we identify three classes of efficiently-solvable cases
for our setting.
In \S\ref{sec:strongfix}, toward practical optimal solution of instances,  we introduce ``strong fixing'', to enhance the classical technique of reduced-cost fixing. 
In \S\ref{sec:comp}, we present results of computational experiments, demonstrating the value of strong fixing for reducing model size and as a useful tool for solving difficult instances to optimality. 
In \S\ref{sec:cp2}, we propose a mixed-integer second-order cone program (MISOCP) to model the more difficult version  of the problem where the locations of the potential SLSs are not prescribed, they are only restricted to belong to given convex sets. 
In \S\ref{sec:strongfixCP2}, we discuss extensions of strong fixing to our MISOCP. 
In \S\ref{sec:compCP2}, we present results of computational experiments, demonstrating that the MISOCPs, although more difficult to solve than the set-covering problems, can lead to significant savings on the overall cost of the SLS placements. The results demonstrate once more the effectiveness of the strong-fixing procedure in reducing model size.
In \S\ref{sec:conc}, we identify some potential next steps. 

\medskip

\noindent{\bf Some relevant related literature.} There is a huge literature on covering problems,
including approximation algorithms (see, for example, \cite{Takazawa}), and a
wide variety of covering applications
(some with a geometric flavor; see, for example,
\cite{Acharyya,BRISKORN,Mao}). We will not attempt to broadly survey these. Rather, we will focus more exhaustively on highlighting related 
literature on covering from the integer-programming viewpoint, and some related geometric covering problems, although none of this is needed for our development. With regard to covering integer-linear programs and associated polyhedra and cutting planes, some  
important references are the fundamental book \cite{CornBook}, the survey \cite{bentz}, and the more specific papers \cite{Balas1,Balas2,Balas1980a,Balas1980b,BellmoreRatliff,Claussen_2022,Sassano,Sass2,Sass3,Sassano, CornNov,Aguilera,Lamothe,Argiroffo,Pashkovich,Bodur}. Finally, we wish to mention a couple of old papers, while not focused on integer programming approaches, do consider some interesting algorithmic ideas:
\cite{BEASLEY198785,Lawler}.

%%%%%%%%%%%%%%%%

\section{Covering edges with a subset of a  finite set of balls: the \ref{scp} problem}\label{sec:covergiven}
We begin with a formally defined geometric optimization problem.
Let $G$ be a straight-line embedding of a graph in $\mathbb{R}^d$, $d\geq 1$ (although our main interest is $d=2$, with $d=3$ possibly also having some applied interest), where we denote  the vertex set of $G$ by $\mathcal{V}(G)$,
and the edge set of $G$ by $\mathcal{I}(G)$, which is a finite set of intervals, which we regard as \emph{closed}, thus containing its end vertices. 
Note that an interval can be a single point (even though this might not be useful for our motivating application). 
We are further given a finite set $N$ of $n$ points in $\mathbb{R}^d$, a weight function $w:N\rightarrow \mathbb{R}_{++}$\,, and  covering radii $r:N\rightarrow \mathbb{R}_{++}$ (we emphasize that points in $N$ may have differing covering radii). 
A point $x\in N$ $r(x)$-covers all points in the $r(x)$-ball 
$B(x,r(x)):=\{ y\in \mathbb{R}^d ~:~ \| x-y \|_2 \leq r(x)\}$. 
A subset $S\subset N$ $r$-covers $G$ if every point $y$ in every edge $I\in \mathcal{I}(G)$ is $r(x)$-covered by some point $x\in S$.
We may as well assume, for feasibility, that $N$ $r$-covers $G$. 
Our goal is to find a minimum $w$-weight $r$-covering of $G$. 

Connecting this geometric problem with our motivating application,
we observe that any realistic road network can be approximated 
to arbitrary precision by a  straight-line embedded graph,
using extra vertices, in addition to road junctions; this is just the standard technique of piecewise-linear approximation of curves. 
The point set $N$ corresponds to the set of potential SLSs. 
In our application, a constant radius for each SLS is rather natural,
but our methodology does not require this. We also allow for 
cost to depend on SLSs, which can be natural if sites are rented, for example. 

We note that the intersection of $B(x,r(x))$ and an edge $I\in \mathcal{I}(G)$ is a closed subinterval (possibly empty) of $I$
which we denote by $I(x,r(x))$.
Considering the nonempty $I(x,r(x))$ as $x\in N$ varies, we get a finite collection $C(I)$ of nonempty closed sub-intervals of $I$.  
The collection $C(I)$ induces a finite collection $\mathcal{C}(I)$ of maximal closed subintervals such that every point $y$ in every 
subinterval in $\mathcal{C}(I)$ is $r$-covered by  precisely the same subset of $N$.  
$\mathcal{C}(I)$ is the natural set of subintervals created from the interval $I$ as it is cut up by the boundaries of the balls intersecting it. 

With all of this notation, we can re-cast  the problem of finding a minimum $w$-weight $r$-covering of $G$ as the 0/1-linear optimization problem 
\begin{align*}
    \min~ &\sum_{x\in N}  w(x) z(x)\\ 
   \mbox{s.t.}~ &\sum_{\substack{x\in N ~:~ \\ J \subset I(x,r(x))}} 
   z(x) \geq 1, ~\forall~ J \in \mathcal{C}(I),~ I\in \mathcal{I}(G);\tag{CP}\label{CP1}\\
   & z(x)\in \{0,1\},~ \forall~ x\in N,
\end{align*}
where $z(x)$ is an indicator variable for choosing the ball indexed by the point $x\in N$.
The constraints simply enforce that each closed sub-interval is covered by some selected ball.
We have $|\mathcal{C}(I)|\leq 1+2n$, for each edge 
$I\in \mathcal{I}(G)$, because each of the $n$ balls can intersect each edge at most once.
Therefore, the
number of covering constraints, which we will denote by $m$,  is
at most $(1+2n)|\mathcal{I}(G)|$. 
Of course we can view this formulation in matrix format as 
\begin{equation}\label{scp}\tag{SCP}
\min\{w^\top z ~:~ Az\geq \mathbf{e},~ 
z\in \{0,1\}^n\}
\end{equation}
for an appropriate 0/1-valued $m\times n$ matrix $A$,
and it is this view that we mainly work with in what follows.

The problem is already NP-Hard for $d>1$, when all balls have identical radius,
the weights are all unity, $N=\mathcal{V}(G)$,
and the edges of $G$ are simply the points in $N$; see \cite[Thm. 4]{FOWLER1981133}.
Of course, this type of graph (with only degenerate edges) is not directly relevant to our motivating application, and anyway we are 
aiming at exact algorithms for practical instances of moderate size. 

\section{What kind of constraint matrices for \ref{scp} can we get?}\label{sec:matrices}

There is a big theory on when 
set-covering LPs have 
integer optima (for all objectives).
It is the theory of ``ideal'' matrices; see \cite{CornBook}.  A 0/1 matrix is \emph{balanced} if it does not contain a square submatrix of
odd order with two ones per row and per column. A 0/1 TU (``totally unimodular''; see \cite[p. 540]{NWbook}, for example) matrix cannot have such a submatrix (which has determinant $\pm2$),
so 0/1 TU matrices are a subclass of 0/1 
balanced matrices.

Berge \cite{Berge} showed that, if $A$ is balanced, then both the packing and covering systems associated with $A$ 
have integer vertices. Fulkerson, Hoffman, and Oppenheim  \cite{Fulkerson1974}
showed that, if $A$ is balanced, then the covering system is TDI (``totally dual integral''; see \cite[p. 537]{NWbook}, for example). 
So balanced 0/1 matrices are a subclass of ideal 0/1 matrices.

We can observe that the matrices that can arise for us are not generally balanced, already for a simple example, depicted in Fig. \ref{fig:5star}; the drawing is for $n=5$, but it could have been for any 
odd  $n\geq 3$. 
The edges of the graph are indicated with red. The points of $N$ are at the midpoints of the black lines (which are not themselves edges). The (green) circle for each point of $N$ goes through the center. It is easy to see that edges are not subdivided by circles, and each circle covers a paper of edges in a cyclic fashion. The constraint matrix of the covering problem is an odd-order (5 in this case) 0/1 matrix violating the definition of balanced.
In fact, the matrix is not even ideal as the
covering LP has a
fractional 
extreme point (in fact with all components equal to 
$\frac{1}{2}$).

\begin{figure}[!ht]

\centering

$\vcenter{\hbox{\includegraphics[height=1.2in]{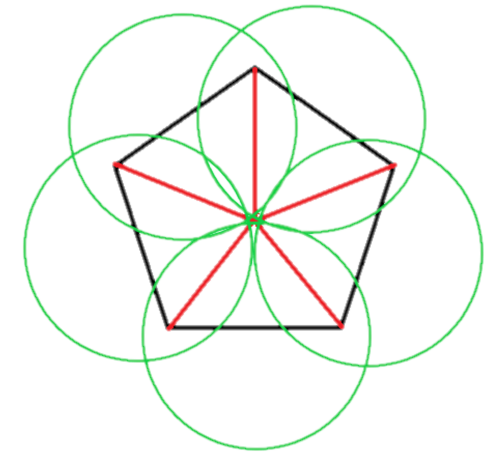}}}$
\qquad
$
\mathcal{C}^2_5 := 
\begin{pNiceMatrix}
%\CodeBefore [create-cell-nodes]
%\Body
1 & 1 & 0 & 0 & 0 \\
0 & 1 & 1 & 0 & 0 \\
0 & 0 & 1 & 1 & 0 \\
0 & 0 & 0 & 1 & 1 \\
1 & 0 & 0 & 0 & 1 \\
\end{pNiceMatrix}
$

\caption{Instance with a non-ideal covering matrix}\label{fig:5star}
\end{figure}

We can also get a counterexample to idealness for the covering matrix with respect to  unit-grid graphs. 
In particular, it is well known that the  ``circulant matrix''
of Fig. \ref{fig:C38matrix} is non-ideal, see \cite{CornNov}.

\tikzset{highlight/.style={red!60!white, ultra thick}}

\begin{figure}[!ht]
\centering
\[
\mathcal{C}^3_8 := 
\begin{pNiceMatrix}
\CodeBefore [create-cell-nodes]
\Body
    1 & 1 & 1 & 0 & 0 & 0 & 0 & 0 \\
    0 & 1 & 1 & 1 & 0 & 0 & 0 & 0 \\
    0 & 0 & 1 & 1 & 1 & 0 & 0 & 0 \\
    0 & 0 & 0 & 1 & 1 & 1 & 0 & 0 \\
    0 & 0 & 0 & 0 & 1 & 1 & 1 & 0 \\
    0 & 0 & 0 & 0 & 0 & 1 & 1 & 1 \\
    1 & 0 & 0 & 0 & 0 & 0 & 1 & 1 \\
    1 & 1 & 0 & 0 & 0 & 0 & 0 & 1 \\
\end{pNiceMatrix}
\]
\caption{Non-ideal circulant matrix}\label{fig:C38matrix}
\end{figure}
Now, we can realize this matrix from an
8-edge unit-grid graph (or you may prefer to see it as a 4-edge unit-grid graph), see Fig. \ref{fig:C38}, and eight well-designed
covering disks, each covering an ``L'', namely 
$\{a,b,c\}$,
$\{b,c,d\}$,
$\{c,d,e\}$,
$\{d,e,f\}$,
$\{e,f,g\}$,
$\{f,g,h\}$,
$\{g,h,a\}$,
$\{h,a,b\}$.
As a sanity check, referring to Fig. \ref{fig:fivematrix}, 
 we see that $\mathcal{C}^3_8$ is
not balanced.

\begin{figure}[!ht]
\centering
\includegraphics[height=2.6in]{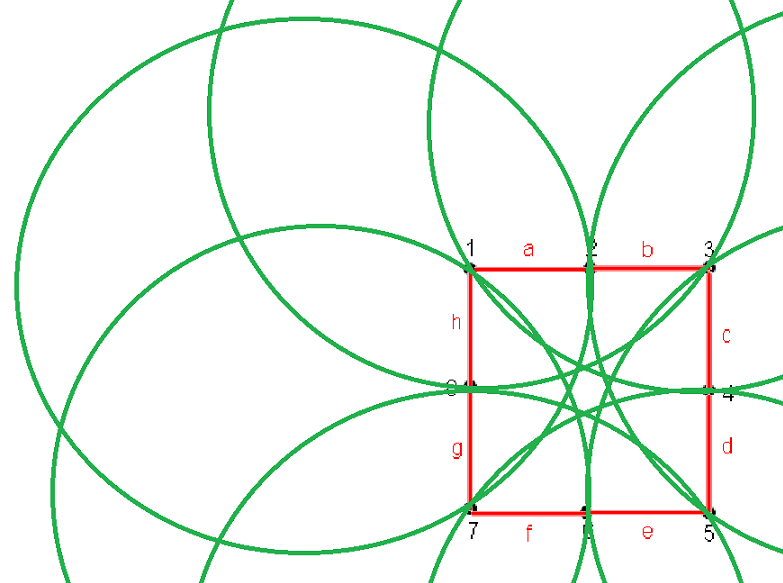}
\caption{Yields the non-ideal covering matrix $\mathcal{C}^3_8$\,.}\label{fig:C38}
\end{figure}

\tikzset{highlight/.style={red!60!white, ultra thick}}

\begin{figure}[!ht]
\centering
\[
\begin{pNiceMatrix}
\CodeBefore [create-cell-nodes]
 \tikz \draw[highlight] (2-1.west) -- (2-8.east) ;
 \tikz \draw[highlight] (5-1.west) -- (5-8.east) ;
 \tikz \draw[highlight] (7-1.west) -- (7-8.east) ;
 \tikz \draw[highlight] (1-2.north) -- (8-2.south) ;
 \tikz \draw[highlight] (1-5.north) -- (8-5.south) ;
 \tikz \draw[highlight] (1-7.north) -- (8-7.south) ;
\Body
    1 & 1 & 1 & 0 & 0 & 0 & 0 & 0 \\
    0 & 1 & 1 & 1 & 0 & 0 & 0 & 0 \\
    0 & 0 & 1 & 1 & 1 & 0 & 0 & 0 \\
    0 & 0 & 0 & 1 & 1 & 1 & 0 & 0 \\
    0 & 0 & 0 & 0 & 1 & 1 & 1 & 0 \\
    0 & 0 & 0 & 0 & 0 & 1 & 1 & 1 \\
    1 & 0 & 0 & 0 & 0 & 0 & 1 & 1 \\
    1 & 1 & 0 & 0 & 0 & 0 & 0 & 1 \\
\end{pNiceMatrix}
\]
\caption{$\mathcal{C}^3_8$ is
not balanced.}\label{fig:fivematrix}
\end{figure}

\section{Solvable cases of \ref{scp}}\label{sec:solveable}

As 0/1 covering linear programs are NP-hard in the worst case, it is natural to consider what cases of our problem are efficiently solvable. In what follows, we present three solvable cases of \ref{scp}.

\subsection{When \texorpdfstring{$G$}{G} is a unit-grid graph in \texorpdfstring{$\mathbb{R}^d$, $d\geq 2$, $N\subset  \mathcal{V}(G)$, and 
 $1\leq r(x)<\sqrt[d]{5/4}$, for all $x\in N$}{}}\label{sec:unit}
 
 Briefly, a \emph{unit-grid graph}
 is a finite subgraph of the standard integer lattice graph.
 For this case, we will observe that our minimum-weight $r$-covering problem on $G$ is equivalently an ordinary minimum-weight vertex covering problem on $G$, with vertices
 in $\mathcal{V}(G)\setminus N$ disallowed. 
 Choosing a vertex $x$ as a covering vertex fully $r$-covers all of its incident edges
 (because $r(x)\geq 1$), but it  will not $r$-cover the midpoint of any other edge 
 (because $r(x)<\sqrt[d]{5/4}$, which is the minimum distance between $x$ and the midpoint of an edge that is not incident to $x$). The only way to $r$-cover a midpoint of an edge is to choose one of its end-points, in which case the entire edge is $r$-covered (because $r(x)\geq 1$). 
 Fig. \ref{fig:grid} illustrates this in $\mathbb{R}^2$: 
 If the disk centered at the cyan {\color{cyan}\Large$\times$} is selected, 
 it fully covers the three adjacent edges (also indicated in cyan), and it does not cover the mid-point of any other edge
 (we indicate the mid-points that are just out of reach with cyan {\color{cyan}$\bullet$}'s). The only way to cover the mid-point of any edge, and then in fact cover the entire edge, is to pick a disk centered at one of its end-points.

\begin{figure}[ht!]
 \begin{center}
\begin{tikzpicture}
\bf\color{black}\tiny
\node[circle, text=cyan, minimum size=1cm]    (e1)  at ( -1.3, -1.18 ) {{\Large$\times$}};

\draw[red, thick] (-5.3,-1.2) -- (2.7,-1.2);
\draw[red, thick] (-5.3,-1.2) -- (-5.3,0.8);
\draw[red, thick] (-3.3,-5.2) -- (-3.3,0.8);
\draw[red, thick] (-3.3,-3.2) -- (-1.3,-3.2);
\draw[red, thick] (-1.3,-3.2) -- (-1.3,-1.2);
\draw[red, thick] (-3.3,0.8) -- (-1.3,0.8);
\draw[red, thick] (0.7,0.8) -- (0.7,-5.2);
\draw[red, thick] (-1.3,-5.2) -- (0.7,-5.2);
\draw[red, thick] (2.7,-1.2) -- (2.7,-3.2);

\draw[cyan, thick] (-3.3,-1.16) -- (-1.25,-1.16);
\draw[cyan, thick] (-3.3,-1.24) -- (-1.25,-1.24);

\draw[cyan, thick] (-1.25,-1.16) -- (0.7,-1.16);
\draw[cyan, thick] (-1.25,-1.24) -- (0.7,-1.24);

\draw[cyan, thick] (-1.25,-1.16) -- (-1.25,-3.2);
\draw[cyan, thick] (-1.35,-1.16) -- (-1.35,-3.2);
\draw[color=cyan, fill=cyan, thick](-3.3,-0.2) circle (0.05);
\draw[color=cyan, fill=cyan, thick](0.7,-0.2) circle (0.05);
\draw[color=cyan, fill=cyan, thick](-3.3,-2.2) circle (0.05);
\draw[color=cyan, fill=cyan, thick](0.7,-2.2) circle (0.05);
\draw[color=cyan, fill=cyan, thick](-2.3,-3.2) circle (0.05);
\draw[color=cyan, fill=cyan, thick](-2.3,0.8) circle (0.05);

\draw[color=black, fill=black, thick](-5.3,-1.2) circle (0.05);
\draw[color=black, fill=black, thick](-3.3,-1.2) circle (0.05);
\draw[color=black, fill=black, thick](-1.3,-1.2) circle (0.05);
\draw[color=black, fill=black, thick](0.7,-1.2) circle (0.05);
\draw[color=black, fill=black, thick](2.7,-1.2) circle (0.05);
\draw[color=black, fill=black, thick](-5.3,0.8) circle (0.05);
\draw[color=black, fill=black, thick](-3.3,0.8) circle (0.05);
\draw[color=black, fill=black, thick](-1.3,0.8) circle (0.05);
\draw[color=black, fill=black, thick](0.7,0.8) circle (0.05);
\draw[color=black, fill=black, thick](-3.3,-3.2) circle (0.05);
\draw[color=black, fill=black, thick](-1.3,-3.2) circle (0.05);
\draw[color=black,fill=black, thick](0.7,-3.2) circle (0.05);
\draw[color=black, fill=black, thick](2.7,-3.2) circle (0.05);
\draw[color=black, fill=black, thick](-3.3,-5.2) circle (0.05);
\draw[color=black, fill=black, thick](-1.3,-5.2) circle (0.05);
\draw[color=black, fill=black, thick](0.7,-5.2) circle (0.05);

\end{tikzpicture}
\caption{grid graph}\label{fig:grid}
\end{center}
\end{figure}
%%%%

The efficient solvability easily follows, because unit-grid graphs are bipartite, 
 and the ordinary formulation of minimum-weight vertex covering 
 (with variables corresponding to vertices in $\mathcal{V}(G)\setminus N$ set to 0) has a totally-unimodular constraint matrix; so the problem is efficiently solved by linear optimization. 
 
\subsection{When \texorpdfstring{$G$}{G} is a path intersecting each ball on a subpath}

When $G$ is a \emph{straight} path  (in any dimension $d\geq 1$), 
ordering the subintervals of $\mathcal{C}(I)$
naturally, as we traverse the path (see Fig. \ref{fig:straight} for an illustration), we can see that in this case the constraint matrix for \ref{scp} is a (column-wise) consecutive-ones matrix, and hence is totally unimodular (and so \ref{scp} is polynomially solvable is such cases).
In fact, this is true as long as the path (not necessarily straight) intersects $B(x,r(x))$
on at most one subpath, for each $x\in N$. For example, if 
$G$ is monotone in each coordinate, then $G$ enters and leaves each ball  at most once each. 

%%%%%
\begin{figure}[ht]
\begin{center}
\begin{tikzpicture}
\bf\color{black}\tiny
\draw[red, thick] (0,0) -- (10,0);

\draw[color=green!60!blue, thick](1,0) circle (1.8);
\draw[color=green!60!blue, thick](2.25,0) circle (0.5);
\draw[color=green!60!blue, thick](2.7,0) circle (1.8);
\draw[color=green!60!blue, thick](4.5,0) circle (1.1);
\draw[color=green!60!blue, thick](6,0) circle (1.3);
\draw[color=green!60!blue, thick](7.7,0) circle (1.5);
\draw[color=green!60!blue, thick](9.5,0) circle (1);

\draw[color=black, fill=black, thick](0,0) circle (0.05);
\draw[color=black, fill=black, thick](1,0) circle (0.05);
\draw[color=black, fill=black, thick](2,0) circle (0.05);
\draw[color=black, fill=black, thick](2.5,0) circle (0.05);
\draw[color=black, fill=black, thick](4,0) circle (0.05);
\draw[color=black, fill=black, thick](5,0) circle (0.05);
\draw[color=black, fill=black, thick](6.5,0) circle (0.05);
\draw[color=black, fill=black, thick](7,0) circle (0.05);
\draw[color=black, fill=black, thick](9,0) circle (0.05);
\draw[color=black, fill=black, thick](10,0) circle (0.05);
\end{tikzpicture}
\end{center}
\caption{path}\label{fig:straight}
\end{figure}

% %%%%

% Ordinary  minimum-weight vertex cover is polynomially solvable on grid graphs (because they are bipartite, and so the obvious formulation has a totally-unimodular constraint matrix); so we can wonder about whether we can solve our minimum-weight $r$-covering problem efficiently on (grid-embedded) grid graphs. 
% As a motivating and solvable special case,
% suppose that $G$ is a (grid-embedded) grid graph with edges of unit length.
% Now take $N:=\mathcal{V}(G)$ and take $r(x):=1$ for all $x\in N$. 
% Then it is easy to see that choosing a vertex $x$ in our instance fully $r$-covers exactly its incident edges. We also $r$-cover vertices that are exactly one grid step away from $x$ but are not connected to $x$ by edges, but that does not help us toward $r$-covering further edges.
% In this way, we have modeled minimum-weight vertex cover on grid graphs by our problem. Going a 
% bit further, we could as well have chosen  $1\leq r(x)<\sqrt{5}/2$, for all $x\in N$. In this way,
% choosing $x$ will $r$-cover all of its incident edges, and will not $r$-cover the midpoint of any other edge. The only way to $r$-cover a midpoint of an edge is to choose one of its endpoints, in which case the entire edge is $r$-covered. 
% This construction works both ways. That is, given a unit-grid graph $G$ and a nonempty  subset $N\subset \mathcal{V}(G)$, and $1\leq r(x)<\sqrt{5}/2$, for all $x\in N$, we  can find a minimum-weight $r$-covering by solving the associated minimum-weight vertex covering problem (with  vertices in $\mathcal{V}(G)\setminus N$  disallowed).

\subsection{A fork-free set of subtrees}\label{sec:fork}

Given a tree $T$, 
a pair of subtrees $T_1$ and $T_2$ has
a \emph{fork} if there is a path $P_1$ with end-vertices in $T_1$ but not $T_2$\,, and
a path $P_2$ with end-vertices in $T_2$ but not $T_1$\,, such that 
$P_1$ and $P_2$ have a vertex in common.
We consider the problem of finding a minimum-weight covering of a tree by a given set of subtrees; see \cite[called problem ``$C_0$'']{BaranyEdmondsWolsey1986}. This problem  admits a polynomial-time algorithm when the family of subtrees is fork free, using some problem transformations and then a recursive algorithm. 
This problem and algorithm is relevant to our situation when: 
$G$ is a tree,  each ball intersects $G$ on a subtree (easily checked), and the set of these subtrees is fork free (easily checked). 
A simple  special case is the situation considered in 
\S\ref{sec:unit}. 
%Much more broadly, 
%if the degrees of a tree are $\leq 3$, then any family of subtrees is  fork-free.

Fig. \ref{fig:subtrees} illustrates a fork-free situation arising in our setting.
The four disks cover the subtrees (indicated by their edge-sets): $\{1,2,3\}$, $\{1,2,3,4\}$, $\{3,4,5,6\}$, and $\{5,6,7,8,9,10,11\}$. It is easy (but a bit tedious) to check that this set of subtrees is fork free.

%%%%%%%%%%%%%%%%%%%%%%%%%%%%%%%%%
%%%%%

\begin{figure}[ht]
\begin{center}
\begin{tikzpicture}
\bf\color{black}\tiny
\draw[red, thick] (0,0) -- (0,1.41);
\draw[red, thick] (0,0) -- (-1,-1);
\draw[red, thick] (0,0) -- (1,-1);

\draw[red, thick] (1,-1) -- (2.0,0);
\draw[red, thick] (1,-1) -- (2,-2);
\draw[red, thick] (2,0) -- (3.5,0.3);
\draw[red, thick] (2,0) -- (4.3,-0.6);
\draw[red, thick] (2,-2) -- (4.4,-1.3);
\draw[red, thick] (2,-2) -- (2.6,-2.9);
\draw[red, thick] (2.6,-2.9) -- (4,-2.3);

\node[circle, text=red, minimum size=0.1cm]    (e1)  at ( -0.1, 0.5 ) {1};
\node[circle, text=red, minimum size=0.1cm]    (e1)  at ( -0.5,-0.35 ) {2};
\node[circle, text=red, minimum size=0.1cm]    (e1)  at ( 0.19,-0.05 ) {3};
\node[circle, text=red, minimum size=0.1cm]    (e1)  at ( 0.68,-0.55 ) {4};
\node[circle, text=red, minimum size=0.1cm]    (e1)  at ( 1.55,-0.3 ) {5};
\node[circle, text=red, minimum size=0.1cm]    (e1)  at ( 1.55,-1.4 ) {6};
\node[circle, text=red, minimum size=0.1cm]    (e1)  at ( 2.6,0.25 ) {7};
\node[circle, text=red, minimum size=0.1cm]    (e1)  at ( 3.2,-0.18 ) {8};
\node[circle, text=red, minimum size=0.1cm]    (e1)  at ( 3.2,-1.53 ) {9};
\node[circle, text=red, minimum size=0.1cm]    (e1)  at ( 2.4,-2.3 ) {10};
\node[circle, text=red, minimum size=0.1cm]    (e1)  at ( 3.2,-2.47 ) {11};

\draw[color=green!60!blue, thick](-1,0.5) circle (1.5);
\draw[color=green!60!blue, thick](0,0) circle (1.41);
\draw[color=green!60!blue, thick](1,-1) circle (1.41);
\draw[color=green!60!blue, thick](2.7,-1.2) circle (1.7);

\draw[color=black, fill=black, thick](0.28,-0.28) circle (0.05);
\draw[color=black, fill=black, thick](0,0) circle (0.05);
\draw[color=black, fill=black, thick](0,1.41) circle (0.05);
\draw[color=black, fill=black, thick](-1,-1) circle (0.05);
\draw[color=black, fill=black, thick](1,-1) circle (0.05);
\draw[color=black, fill=black, thick](2.0,0)circle (0.05);
\draw[color=black, fill=black, thick](2,-2)circle (0.05);
\draw[color=black, fill=black, thick](3.5,0.3)circle (0.05);
\draw[color=black, fill=black, thick](4.3,-0.6)circle (0.05);
\draw[color=black, fill=black, thick](4.4,-1.3)circle (0.05);
\draw[color=black, fill=black, thick](2.6,-2.9)circle (0.05);
\draw[color=black, fill=black, thick](4,-2.3)circle (0.05);
\end{tikzpicture}
\end{center}
\caption{subtrees}\label{fig:subtrees}
\end{figure}

%%%%

Now, it 
is  easy to make a  simple example arising from our situation (which is even a unit-grid graph), where a fork arises; see Fig. \ref{fig:forking}
(The tree is the red graph, and the pair of green covering disks define the two subtrees). 
\begin{figure}[!ht]
\centering
\includegraphics[height=0.9in]{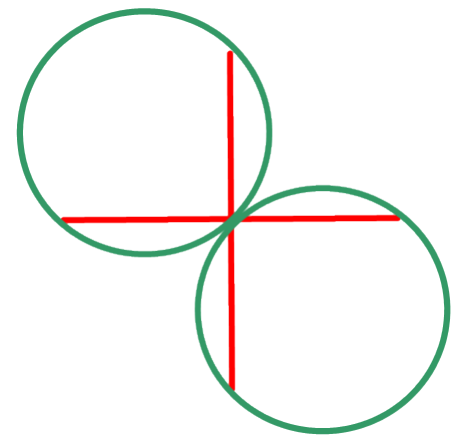}
\caption{A forking configuration}\label{fig:forking}
\end{figure}
So the algorithm from 
\cite{BaranyEdmondsWolsey1986} does not apply to this example.  On the other hand, the constraint matrix of this instance is balanced, so this is an easy instance (solvable by linear programming).  

Referring back to Fig. \ref{fig:5star},
we also find forks. 
Also see \cite[Sec. 5]{BaranyEdmondsWolsey1986} which raises the interesting question on the relationship between \emph{totally}-balanced matrices\footnote{a 0/1 matrix is \emph{totally balanced} if it has no square submatrix \emph{(of any order)}
 with two ones per row and per column, thus a subclass of balanced 0/1 matrices; see \cite{totallybalanced}.} and covering matrices of fork-free families. These notions are incomparable; 
\cite{BaranyEdmondsWolsey1986}  has a very simple example that is fork-free but not \emph{totally} balanced. But we can even get fork-free coming from our context and 
not balanced, returning again to 
our example of Fig. \ref{fig:C38} (it is not a tree, but we can break an edge).

\section{Strong fixing for \ref{scp}}\label{sec:strongfix}

A well-known result for mixed-integer nonlinear-minimization formulations is that if we have a good upper bound on its optimal objective value and a feasible dual solution to a convex relaxation of the problem, where the dual variable associated with a primal inequality constraint satisfies an easily-checked condition, we can immediately infer that the constraint is satisfied as an equation in every discrete optimal solution; in other words, we can fix the slack variable associated with the inequality to zero. In this section, we discuss this result and present a \emph{strong-fixing} procedure, in which we search for feasible dual solutions that lead to the fixing of the greatest number of variables. This result brings a fundamental improvement to the branch-and-bound approach for discrete optimization problems based on convex relaxations.

We  consider the linear relaxation of \ref{scp}, that is  
% \begin{equation}\label{p}\tag{P}
$
\min\{w^\top z ~:~ Az\geq \mathbf{e},~ 
z\geq 0\},
$
% \end{equation}
and the associated dual problem
\begin{equation}\label{d}\tag{D}
\max\{u^\top \mathbf{e} ~:~ u^\top A\leq  w^\top,~ 
u\geq 0\}.
% \begin{array}{lll}
%  &\max& u^\top \mathbf{e}\\
%     &\mbox{s.t.} &u^\top A\leq w^\top,\\
%     &&u\ge 0.
% \end{array}
\end{equation}
An optimal solution of \ref{d} is commonly  used in the application of \emph{reduced-cost fixing}, see, e.g. \cite{reducedcostfixing}, a classical technique in integer programming that uses upper bounds on the optimal solution values of minimization problems for inferring variables whose values can be fixed while preserving the optimal solutions. The well-known technique is based on Thm. \ref{thm:fix_dopt}  (see proof in \ref{app:fixSCP0}).

\begin{theorem}\label{thm:fix_dopt}
Let UB be the objective-function value of a feasible solution for  {\rm\ref{scp}}, and let $\hat u$ be a feasible solution for  {\rm\ref{d}}.
Then, for every optimal solution $z^\star$ for {\rm\ref{scp}}, we have:
\vspace{-5pt}
\begin{align}
     &z_j^\star \leq  \textstyle \left\lfloor
  \frac{UB-\hat{u}^\top\mathbf{e}}{w_j - \hat{u}^\top A_{\cdot j}}\right\rfloor ,\quad ~\forall\; j \in \{1,\ldots,n\} \text{ such that } w_j - \hat{u}^\top A_{\cdot j}>0.\label{ineq-fixing}
     \end{align}
 \end{theorem}
 
For a given $j\in \{1,\ldots,n\}$, we should have that the right-hand side in \eqref{ineq-fixing} equal to 0 to be able to fix the variable $z_j$ at 0 in \ref{scp}. Equivalently, we should have
$
 w_j+ \hat{u}^\top(\mathbf{e} -  A_{\cdot j})> UB\,.
$
% We consider a permutation $\sigma$ of the indices in $N$, defined by some ordering of the indices. 

We observe that any feasible solution $\hat{u}$ can be used in \eqref{ineq-fixing}. Then, for all $j\in \{1,\ldots,n\}$, we propose the solution of 
\begin{align}
    \mathfrak{z}_{j}^{\mbox{\tiny SCP(0)}}:=w_j+\max\{u^\top (\mathbf{e} -  A_{\cdot j})~:~ u^\top A\leq  w^\top,~ 
u\geq 0\}. \label{fj}\tag{F$_j^{\mbox{\tiny SCP(0)}}$}
\end{align}
Note that, for each $j\in \{1,\ldots,n\}$, if there is a feasible solution $\hat{u}$ to \ref{d} that can be used in \eqref{ineq-fixing} to fix $z_j$ at 0, then  the optimal solution of \ref{fj} has objective value $\mathfrak{z}_{j}^{\mbox{\tiny SCP(0)}^*}$ greater than $UB$ and can be used as well. 

Now, we note that adding the redundant constraint $z\leq \mathbf{e}$ to the linear relaxation of \ref{scp} would lead to the modified dual problem  
\begin{equation}\label{dplus}\tag{D$^+$}
\max\{u^\top \mathbf{e} - v^\top \mathbf{e}~:~ u^\top A - v^\top \leq  w^\top,~ 
u,v\geq 0\},
\end{equation}
where $v\in\mathbb{R}^n$ is the dual variable corresponding to the redundant constraint. We notice that the vectors $\mathbf{e}$ in the objective are both vectors of ones, but with different dimensions according to the dimensions of $u$ and $v$. Analogously to Thm. \ref{thm:fix_dopt}, we can establish the following result (see proof in \ref{app:fixSCP}).

\begin{theorem}\label{thm:fix_doptat1}
Let UB be the objective-function value of a feasible solution for  {\rm\ref{scp}}, and let $(\hat u,\hat v)$ be a feasible solution for  {\rm\ref{dplus}}.
Then, for every optimal solution $z^\star$ for {\rm\ref{scp}}, we have:
\vspace{-5pt}
\begin{align}
% z_j^\star \leq  \textstyle \left\lfloor
%   \frac{UB-\hat{u}^\top\mathbf{e} + \hat{v}^\top\mathbf{e}}{w_j - \hat{u}^\top A_{\cdot j} + \hat{v}_j}\right\rfloor, &\forall j \in \{1,\ldots,n\} \text{ such that } w_j - \hat{u}^\top A_{\cdot j} + \hat{v}_j>0,\\
z_j^\star = 1,\quad ~\forall\;  j\in \{1,\ldots,n\} \mbox{ such that }   \hat{u}^\top\mathbf{e} - \hat{v}^\top\mathbf{e} + \hat{v}_j >  UB. \label{fix1dplus}
     \end{align}
 \end{theorem}

 The redundant  constraint $z\leq\mathbf{e}$ can be useful if we search for the \emph{best feasible solution} to \ref{dplus} that can be used in \eqref{fix1dplus} to fix $z_j$ at 1. With this purpose,  we propose the solution of
\begin{align}
    \mathfrak{z}_{j}^{\mbox{\tiny SCP(1)}}:=\max\{u^\top\mathbf{e}-v^\top\mathbf{e} + v_j~:~ u^\top A- v^\top \leq  w^\top,~ 
u,v\geq 0\}, \label{fj1}\tag{F$_j^{\mbox{\tiny SCP(1)}}$}
\end{align}
for all $j\in\{1,\ldots,n\}$.
If the value of the optimal solution  of \ref{fj1} is greater than $UB$, we can fix $z_j$ at 1.

We call \emph{strong fixing}, the procedure that fixes all possible variables in \ref{scp} at 0 and 1,
in the context of Thm. \ref{thm:fix_dopt} and Thm. \ref{thm:fix_doptat1},
 by using  a given upper bound $UB$ on the optimal solution value of \ref{scp}, and solving  all problems \ref{fj} and \ref{fj1},  for $j\in\{1,\ldots,n\}$.

% {\color{blue}
% To do: 
% For each instance considered in the spreadsheet we already have, solve problems \ref{fj}, for all $j=1,\ldots,n$, and save the size of the set of all variables that could be fixed by all the $n$ dual solutions computed; the mean and standard deviation of the elapsed times to solve all the $n$ problems \ref{fj}, $j=1,\ldots,n$.}

% Our goal when defining $\sigma$ is to order the problems F$_{\sigma(j)}$ in such a way that the overall time to solve them is the least possible. We aim to reduce the time for solving the problems by using the optimal solution of  F$_{\sigma(j-1)}$ as a starting point in the resolution of F$_{\sigma(j)}$, for all $j>1$, $j\in N$. 1 possible permutation $\sigma$ to consider is such that $\sum_{i=1}^m A_{i\sigma(1)}\geq \sum_{i=1}^m A_{i\sigma(2)}\geq \cdots\geq \sum_{i=1}^m A_{i\sigma(n)}$.  

\section{Computational experiments for \ref{scp}}\label{sec:comp}

 We have  implemented a framework 
 %described in \S\ref{subsec:gen} 
 to generate random instances of \ref{CP1} and formulate them as the set-covering problem \ref{scp}.
 We solve the instances applying the following procedures in the given order.
\begin{itemize}
    \item[{\hypertarget{(a)}{(a)}}] Reduce the number of constraints in \ref{scp} by eliminating dominated rows of the associated constraint matrix $A$.
    \item[{\hypertarget{(b)}{(b)}}]  Fix variables in the reduced \ref{scp}, applying \emph{reduced-cost fixing}, i.e., using  Thm. \ref{thm:fix_dopt}, taking $\hat{u}$ as an optimal solution of  \ref{d}. If it was possible to fix variables, reapply \hyperlink{(a)}{(a)} to reduce the number of constraints further.
    \item[{\hypertarget{(c)}{(c)}}] Apply  \emph{strong fixing} (see \S\ref{sec:strongfix}). In case it was possible to fix variables, reapply procedure \hyperlink{(a)}{(a)} to reduce the number of constraints in the remaining problem. 
    \item[{\hypertarget{(d)}{(d)}}] \label{procd} Solve the last problem obtained with \texttt{Gurobi}.
\end{itemize}

Our implementation is in \texttt{Python}, using  \texttt{Gurobi}  v. 10.0.2. We ran  \texttt{Gurobi} with one thread per core, default parameter settings (with the presolve option on).
We ran  on
an 8-core machine (under Ubuntu): Intel i9-9900K
CPU  running at 3.60GHz, with  32
GB of memory. %\jon{mention one thread per core}

\subsection{Generating test instances}\label{subsec:gen}

In Alg. \ref{alg:gen}, we show how we construct random instances for \ref{CP1}, for a given number of points $n$,   a given interval $[R_{\min},R_{\max}]$ in which  the covering radii for the points must be, and a given number $\nu$ of nodes in the graph $G=(V,E)$ to be covered.  We construct the graph $G$ with node-set $V$ given by  
%$\nu:=0.03n$ 
$\nu$ points randomly generated in the unit square in the plane, and the edges in $E$ initially given by the edges of the minimum spanning tree (MST) of $V$, which guarantees connectivity of $G$. Finally, we compute the Delaunay triangulation of the $\nu$ points in $V$, and add to $E$  the edges from the triangulation that do not belong to its convex hull.  All other details of the instance generation can be seen in Alg. \ref{alg:gen}. In Fig. \ref{fig:graphG}, we show an example of the MST of $V$ and of the graph $G$.

\begin{figure}[ht]
\centering
\includegraphics[width=0.5\textwidth]{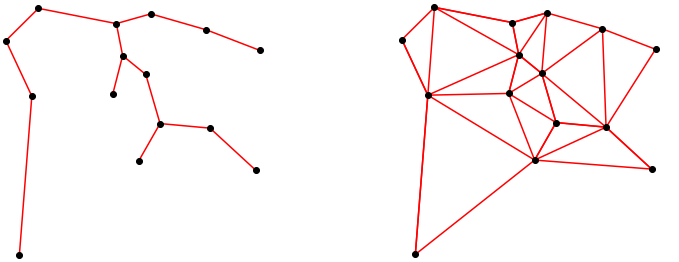}
\caption{Constructing an instance of \ref{CP1}:  MST of $V$ (left) and graph $G$ (right)}\label{fig:graphG}
\end{figure}

\begin{algorithm}[!ht]
	\footnotesize{
		\KwIn{  $n$, $\nu$,  $R_{\min}$, $R_{\max}$ ($0.1 < R_{\min} < R_{\max} < 0.2$).}
		% \KwOut{$m$, $w\in\mathbb{R}^n$, $A\in\{0,1\}^{m\times n}$.}
  \KwOut{instance of \ref{CP1}/\ref{scp}.}
% let $ Q:=\{ [0,1] \times[0,1]\}$\;
 randomly generate a set $V$ of $\nu$
 %$\nu:=0.03n$ 
 points in $ Q:=\{ [0,1] \times[0,1]\}$\;
 use the Python package \texttt{networkx} to compute the edge set $E_{\mbox{\tiny{ST}}}$ of an MST of $V$, letting the weight for edge $(i,j)$  as the Euclidean distance between $i$ and $j$\;
 %,  and denote  the set of edges of the MST by   $E_{\mbox{\tiny{ST}}}$\,\label{step3}\;
 compute the Delaunay triangulation of $V$, and denote the subset of its edges that are not in the convex hull by $E_{\mbox{\tiny{DT}}}$\,\;
 let $G=(V,E)$, where $E:=E_{\mbox{\tiny{ST}}}\cup E_{\mbox{\tiny{DT}}}$\,\;
 randomly generate a set $N$ of $n$  points $x^k$ in $Q$, $k=1,\ldots,n$\label{step5}\;
 randomly generate $r_k$  in $[R_{\min},R_{\max}]$, $k=1,\ldots,n$\;
  randomly generate $w_k$ in $[0.5r_k^2\,,1.5r_k^2]$, $k=1,\ldots,n$\; 
 let $c_k$ be the circle centered at $x^k$ with radius $r_k$\,, for $k=1,\ldots,n$\;
 for each $e \in E$, compute the intersections (0,1, or 2) of $c_k$ and $e$\;
 compute all the subintervals defined on each edge by the intersections, and let $m$ be the total number of subintervals for all edges\; 
 % construct a $0/1$-valued  $m \times n$ matrix $A$, where $A_{ik}=1$ if subinterval $i$ is covered by circle $c_k$, that is, if  the subinterval is inside the circle\; 
 \caption{Instance generator \label{alg:gen} }
}
\end{algorithm}

We note that the instance constructed by Alg. \ref{alg:gen} may be infeasible, if any part of an edge of $G$ is not covered by any point. In this case, we iteratively increase all radii $r_k$\,, for $k=1,\ldots,n$, by 10\%, until the instance is feasible. 
For feasible instances, we check if there are circles that do not intercept any edges of the graph. If so, we  iteratively increase the radius of each of those circles by 10\%, until they intercept an edge. By this last procedure, we avoid zero columns in the constraint matrices $A$ associated to our instances of \ref{scp}.
%\FloatBarrier
In Fig. \ref{fig:sol}, we represent  the data for an instance of  \ref{CP1} and its optimal solution. In the optimal solution we see  9 points/circles selected.
% , which together cover all edges in the graph. 

\begin{figure}[ht]
%
%\centering
% \includegraphics[width=1.05\textwidth]{sol.jpg}
\includegraphics[width=0.45\textwidth]
{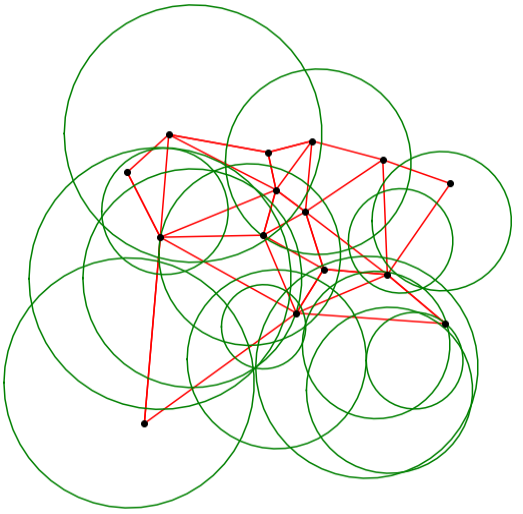}~
\includegraphics[width=0.45\textwidth]{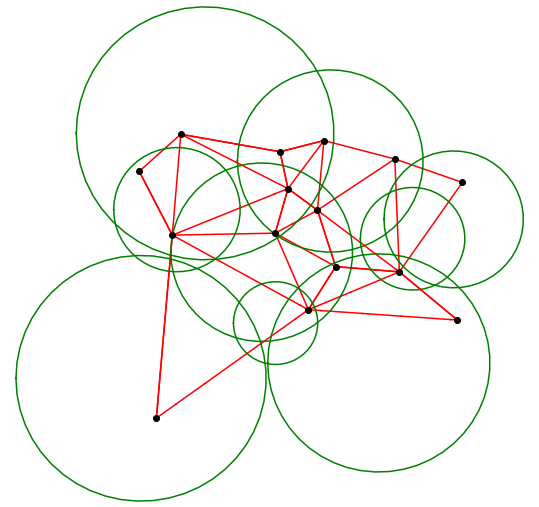}
\caption{Instance data (left) and its optimal solution (right)}\label{fig:sol}
\end{figure}

\subsection{Strong-fixing experiments}

In Table \ref{tab:my-table} we present detailed numerical results for five instances of each problem size considered. We aim at observing the impact of  strong fixing in reducing the size of  \ref{scp}, after having already applied standard reduced-cost fixing and eliminating redundant constraints.  
Although we discussed in \S\ref{sec:strongfix} the strong fixing procedure for \ref{scp}, to fix variables to 0 and 1, we ended up applying only the former in our numerical experiments. Initially, we tried the latter as well, but it did not compensate, as only a small number of variables could be fixed to 1. 
In Table \ref{tab:my-table}, we display the number of rows ($m$) and columns ($n$) of the constraint matrix $A$ after applying each procedure  described in \hyperlink{(a)}{(a}-\hyperlink{(d)}{d)} and their elapsed time (seconds). For all instances we set $\nu=0.03n$; having $\nu$ proportional to $n$ gave us
similar geometry for instances of different sizes, and the constant 0.03
gave us reasonable looking and challenging instances (note that these parameters do not correspond to the small illustrative example in Fig. \ref{fig:sol}). 
In the following, we specify the statistics presented in Table \ref{tab:my-table}. We have under 
\begin{itemize}
    \item \texttt{Gurobi}: the size of the original  $A$ and the time to solve  with \texttt{Gurobi}.
    \item \texttt{Gurobi} presolve: the size of the  matrix $A$ after \texttt{Gurobi}'s presolve is applied and the time to apply it. 
    \item matrix reduction: the size of the matrix $A$ after our procedure to eliminate dominated rows  of the original matrix $A$ is applied and the time to apply it. 
    \item reduced-cost fixing: the size of the matrix $A$ after we apply the result in Thm. \ref{thm:fix_dopt}, taking $\hat u$ as the optimal solution of \ref{d}, and the time to apply it. 
    \item strong fixing: the size of the matrix $A$ after we apply our strong-fixing procedure and the optimal solution of \ref{d}, and the time to apply it. 
    \item \texttt{Gurobi} reduced time: the time to solve the instance of the problem obtained after we apply the strong-fixing procedure. 
\end{itemize}

We see that our own presolve, corresponding to procedures \hyperlink{(a)}{(a}-\hyperlink{(c)}{c)}, is  effective in reducing the size of the problem, and does not lead in general to problems bigger than the ones obtained with \texttt{Gurobi}'s presolve (even though from the increase in the number of variables when compared to the original problem, we see that \texttt{Gurobi}'s presolve implements different procedures, such as sparsification on the equation system after adding slack variables).  We also see that \emph{strong fixing} is very effective in reducing the size of the problem. Compared to the problems to which it is applied, we have an average decrease of $47\%$ in $m$ and $42\%$ in $n$. Of course, our presolve is very time consuming compared to \texttt{Gurobi}'s. Nevertheless, for all instances where \emph{strong fixing} can further reduce a significant number of variables, there is an improvement in the final time to solve the problem, and we note that  all the steps of reduction and fixing can still be further improved. This shows that \emph{strong fixing} is a promising tool to be adopted  in the solution of difficult problems, as is the case of the well-known \emph{strong-branching} procedure.
%implemented in integer programming solvers \cite{??}. 
% Finally, we note that the running time of our procedures could be significantly decreased with a more ambitious implementation. 

\begin{table}[ht]
\resizebox{\textwidth}{!}{%
\begin{tabular}{r|rrr|rrr|rrr|rrr|rrr|r}
\multicolumn{1}{c|}{}&\multicolumn{3}{c|}{\texttt{Gurobi}}&\multicolumn{3}{c|}{\texttt{Gurobi}}&\multicolumn{3}{c|}{matrix}&\multicolumn{3}{c|}{reduced-cost}&\multicolumn{3}{c|}{strong}&\multicolumn{1}{c}{\texttt{Gurobi}}\\[-2pt] 
\multicolumn{1}{c|}{}&\multicolumn{3}{c|}{ }&\multicolumn{3}{c|}{presolve}&\multicolumn{3}{c|}{reduction}&\multicolumn{3}{c|}{fixing}&\multicolumn{3}{c|}{fixing}&\multicolumn{1}{c}{reduced}\\[-2pt]
\multicolumn{1}{c|}{\#}&\multicolumn{1}{c}{$m$}&\multicolumn{1}{c}{$n$}&\multicolumn{1}{r|}{time}&\multicolumn{1}{c}{$m$}&\multicolumn{1}{c}{$n$}&\multicolumn{1}{r|}{time}&\multicolumn{1}{c}{$m$}&\multicolumn{1}{c}{$n$}&\multicolumn{1}{r|}{time}&\multicolumn{1}{c}{$m$}&\multicolumn{1}{c}{$n$}&\multicolumn{1}{r|}{time}&\multicolumn{1}{c}{$m$}&\multicolumn{1}{c}{$n$}&\multicolumn{1}{r|}{time}&\multicolumn{1}{c}{time}\\[2pt]
\hline
&&&&&&&&&&&&&&&\\[-6pt]
% 1  & 7164  & 500  & 0.28      & 372  & 257  & 0.17  & 531  & 500  & 24.49   & 214  & 167  & 0.50   & 12   & 19   & 4.98  & 0.003  \\
% 2  & 4078  & 500  & 0.37      & 427  & 327  & 0.10  & 574  & 500  & 12.08   & 344  & 267  & 0.94   & 49   & 63   & 15.91  & 0.01 \\
% 3  & 3865  & 500  & 0.34      & 399  & 331  & 0.10  & 537  & 500  & 10.22   & 342  & 291  & 0.83   & 46   & 71   & 16.90  & 0.007 \\
% 4  & 4399  & 500  & 0.27      & 412  & 300  & 0.10  & 559  & 500  & 13.15   & 310  & 236  & 0.79   & 6    & 14   & 11.26  & 0.007 \\
% 5  & 4098  & 500  & 0.22      & 446  & 359  & 0.13  & 546  & 500  & 11.92   & 146  & 132  & 0.35   & 16   & 16   & 2.53   & 0.008 \\
1  & 7164  & 500  & 0.28      & 372  & 257  & 0.17  & 531  & 500  & 24.49   & 214  & 167  & 0.50   & 12   & 13   & 5.67  & 0.00  \\
2  & 4078  & 500  & 0.37      & 427  & 327  & 0.10  & 574  & 500  & 12.08   & 344  & 267  & 0.94   & 49   & 52   & 19.18  & 0.00 \\
3  & 3865  & 500  & 0.34      & 399  & 331  & 0.10  & 537  & 500  & 10.22   & 342  & 291  & 0.83   & 46   & 58   & 21.39  & 0.00 \\
4  & 4399  & 500  & 0.27      & 412  & 300  & 0.10  & 559  & 500  & 13.15   & 310  & 236  & 0.79   & 6    & 10   & 11.49  & 0.00 \\
5  & 4098  & 500  & 0.22      & 446  & 359  & 0.13  & 546  & 500  & 11.92   & 146  & 132  & 0.35   & 0   & 0   & 2.57   & 0.00 \\
6  & 13254 & 1000 & 2.90      & 1279 & 1119 & 0.73  & 1728 & 1000 & 136.00  & 1331 & 705  & 15.92  & 244  & 202  & 133.39  & 0.03 \\
7  & 9709  & 1000 & 3.91      & 1263 & 1125 & 0.77  & 1581 & 1000 & 81.86   & 1443 & 837  & 13.33  & 909  & 562  & 484.77 & 2.29 \\
8  & 10847 & 1000 & 2.53      & 1307 & 1140 & 0.80  & 1672 & 1000 & 94.13   & 1351 & 712  & 15.22  & 263  & 196  & 131.05  & 0.06 \\
9  & 9720  & 1000 & 2.81      & 1335 & 1149 & 0.69  & 1700 & 1000 & 84.27   & 1247 & 667  & 12.14  & 209  & 192  & 120.86  & 0.00 \\
10 & 9377  & 1000 & 3.20      & 1300 & 1193 & 0.67  & 1621 & 1000 & 75.76   & 1329 & 764  & 12.17  & 322  & 246  & 161.34 & 0.08 \\
11 & 17970 & 1500 & 29.46     & 2720 & 2222 & 2.75  & 3175 & 1500 & 278.33  & 2862 & 1304 & 62.73  & 1061 & 609  & 507.11 & 2.31 \\
12 & 22022 & 1500 & 65.65     & 2838 & 2208 & 3.01  & 3316 & 1500 & 425.03  & 3115 & 1336 & 68.69  & 1725 & 844  & 651.32 & 9.24 \\
13 & 18716 & 1500 & 42.05     & 2944 & 2219 & 2.86  & 3466 & 1500 & 315.96  & 3311 & 1379 & 82.08  & 2184 & 983  & 646.89 & 24.44 \\
14 & 18464 & 1500 & 14.30     & 2545 & 1988 & 2.44  & 3114 & 1500 & 306.49  & 2856 & 1270 & 59.32  & 1548 & 762  & 457.58 & 6.17 \\
15 & 19053 & 1500 & 33.95     & 2908 & 2207 & 2.82  & 3376 & 1500 & 327.13  & 3162 & 1361 & 72.56  & 1745 & 851  & 614.52 & 7.57 \\
16 & 29227 & 2000 & 170.78    & 4194 & 3037 & 6.59  & 4988 & 2000 & 832.81  & 4691 & 1793 & 177.51 & 2959 & 1221 & 1677.15 & 83.46\\
17 & 29206 & 2000 & 212.07    & 4365 & 3186 & 7.04  & 4993 & 2000 & 823.83  & 4793 & 1828 & 188.71 & 3182 & 1318 & 1789.23 & 102.57\\
18 & 28927 & 2000 & 762.26    & 4120 & 2971 & 7.22  & 5016 & 2000 & 799.22  & 4889 & 1862 & 192.64 & 4060 & 1557 & 1961.38 & 497.09\\
19 & 30746 & 2000 & 6970.90   & 4943 & 3645 & 7.22  & 5243 & 2000 & 842.61  & 5243 & 1998 & 229.23 & 5109 & 1944 & 3602.62 & 7421.27\\
20 & 31606 & 2000 & 1914.49   & 5158 & 3711 & 8.20  & 5378 & 2000 & 885.00  & 5374 & 1997 & 241.15 & 5088 & 1897 & 3665.37 & 1034.66\\
21 & 42137 & 2500 & 30111.51 & 6928 & 4656 & 14.99 & 7606 & 2500 & 1785.99 & 7580 & 2470 & 517.05 & 7377 & 2382 & 8332.57& 26554.56 \\
22 & 41229 & 2500 & 594.63    & 6511 & 4454 & 15.07 & 7454 & 2500 & 1689.45 & 7127 & 2344 & 432.95 & 3979 & 1450 & 4672.70 & 142.18\\
23 & 43100 & 2500 & 614.62    & 6236 & 4209 & 15.16 & 7381 & 2500 & 1856.51 & 7139 & 2313 & 462.52 & 5477 & 1828 & 4528.96 & 494.37\\
% 24 & 42053 & 2500 & 51978.94  & 6654 & 4776 & 16.84 & 7031 & 2500 & 1685.95 & 7031 & 2500 & 433.37 & 6986 & 2481 & 8732.90 & 55957.59\\
24 & 42809 & 2500 & 27394.59  & 7210 & 4836 & 18.89 & 7851 & 2500 & 964.75 & 7810 & 2456 & 179.03 & 7557 & 2362 & 8474.90 & 21279.99\\
25 & 42641 & 2500 & 3594.00   & 6639 & 4723 & 16.43 & 7038 & 2500 & 1615.20 & 7032 & 2497 & 412.83 & 6604 & 2360 & 7047.04 & 6540.81
\end{tabular}%
}
\caption{\ref{scp} reduction experiment (5 instances of each size)}
\label{tab:my-table}
\end{table}
In Table \ref{tab:50inst}, we show the shifted geometric mean with shift parameter 1 (commonly used for aggregated comparisons in integer programming; see \cite{Berthold_Hendel_2021}, for example) for the same statistics presented in Table \ref{tab:my-table}, for 50 instances of each $n=500,1000,1500,2000$  and for 39 instances of $n=2500$ (we generated 50 instances with $n=2500$, but we only consider the 39 instances that could be solved within our time limit of 10 hours). 
The average time-reduction factor for all instances solved for $n=500, 1000, 1500, 2000,2500$ is respectively, $0.01$, $0.13$, $0.55$, $0.86$, and $0.76$.
The time-reduction factor is the ratio of the elapsed time used by \texttt{Gurobi} to solve the problem after we applied the matrix reduction and the strong fixing to the elapsed time used to solve the original problem.  When the factor is less than one,  our presolve led to 
 a problem that could be solved faster than the original problem.
\begin{table}[ht]
\resizebox{\textwidth}{!}{%
\setlength\tabcolsep{4pt}
\begin{tabular}{rrr|rrr|rrr|rrr|rrr|r}
\multicolumn{3}{c|}{\texttt{Gurobi}}&\multicolumn{3}{c|}{\texttt{Gurobi}}&\multicolumn{3}{c|}{matrix}&\multicolumn{3}{c|}{reduced-cost}&\multicolumn{3}{c|}{strong}&\multicolumn{1}{c}{\texttt{Gurobi}}\\[-2pt] 
\multicolumn{3}{c|}{ }&\multicolumn{3}{c|}{presolve}&\multicolumn{3}{c|}{reduction}&\multicolumn{3}{c|}{fixing}&\multicolumn{3}{c|}{fixing}&\multicolumn{1}{c}{reduced}\\[-2pt]
\multicolumn{1}{c}{$m$}&\multicolumn{1}{c}{$n$}&\multicolumn{1}{r|}{time}&\multicolumn{1}{c}{$m$}&\multicolumn{1}{c}{$n$}&\multicolumn{1}{r|}{time}&\multicolumn{1}{c}{$m$}&\multicolumn{1}{c}{$n$}&\multicolumn{1}{r|}{time}&\multicolumn{1}{c}{$m$}&\multicolumn{1}{c}{$n$}&\multicolumn{1}{r|}{time}&\multicolumn{1}{c}{$m$}&\multicolumn{1}{c}{$n$}&\multicolumn{1}{r|}{time}&\multicolumn{1}{c}{time}\\[2pt]
\hline
&&&&&&&&&&&&&&&\\[-6pt] 
% 3324.76 &	500 &	0.22 &	391.18 &	313.38 & 0.10 &	517.59 &	500 &	5.49 &	216.77 &	179.87 &	0.31 &	7.20 &	11.06 &	6.03 &	0.00 \\ 
3324.76 &	500 &	0.22 &	391.18 &	313.38 & 0.10 &	517.59 &	500 &	5.49 &	216.77 &	179.87 &	0.31 &	6.75 &	8.17 &	6.43 &	0.00 \\ 
9819.57 &	1000 &	3.81 &	1379.35 &	1220.71 &	0.82 &	1673.98 &	1000 &	48.84 &	1281.17 &	707.19 &	5.58 &	183.99 &	165.37 &	159.11 &	0.50\\
19028.97 &	1500 &	70.78 &	2908.71 &	2270.53 &	3.42 &	3319.23 &	1500 &	181.51 &	3157.20 &	1377.20 &	29.35 &	1971.15 &	944.04 &	699.93 &	29.68\\
29924.19 & 2000 & 1350.81 & 4721.50 & 3444.67 & 7.05 & 5220.12 & 2000 & 836.22 & 5149.53 & 1949.80 & 236.20 & 4379.65 & 1693.80 & 2253.60 & 1062.38 \\ 
42038.95 &	2500 &	10347.39 &	6740.81 &	4636.68 &	17.39 &	7442.30 &	2500 &	886.65 &	7372.69 &	2447.14 &	163.88 &	6648.62 &	2224.60 &	6868.64 & 7265.90 
\end{tabular}%
}
\caption{ \ref{scp} reduction experiment (shifted geometric mean for 50 instances of each size)}
\label{tab:50inst}
\end{table} 
In Fig. \ref{fig:numinst}, we show the number of instances for each $n=1000,1500,2000,2500$,
%(out of 50 for $n=1000,1500,2000$, and out of 39 for $n=2500$) 
for which the time-reduction factor is not greater than $0.1, 0.2,\ldots,1$.  
 % We also see that for $n=2000$, more than 40\% of the instances had their times reduced to no more than 70\% of the original times, and more than  70\% of the instances had their times reduced.  
 For $n=2500$, for the 39 instances solved,  approximately 50\% of them had their times reduced to no more than 70\% of the original times, and   90\% of them had their times reduced. We do not have a plot for $n=500$ because all of those instances had their times reduced to no more than 10\% of the original times. 

\begin{figure}[ht!]
    \centering
    \begin{subfigure}
        \centering
        \includegraphics[width=0.49\textwidth]{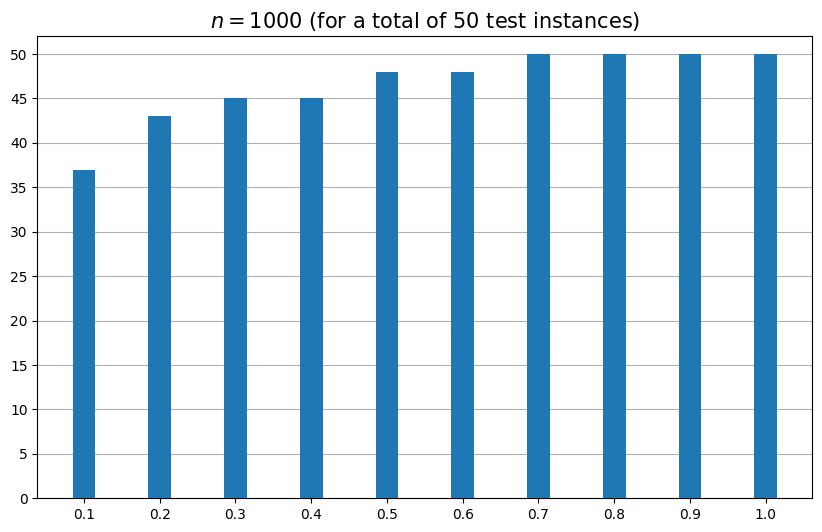}
    \end{subfigure}
    \begin{subfigure}
        \centering
        \includegraphics[width=0.49\textwidth]{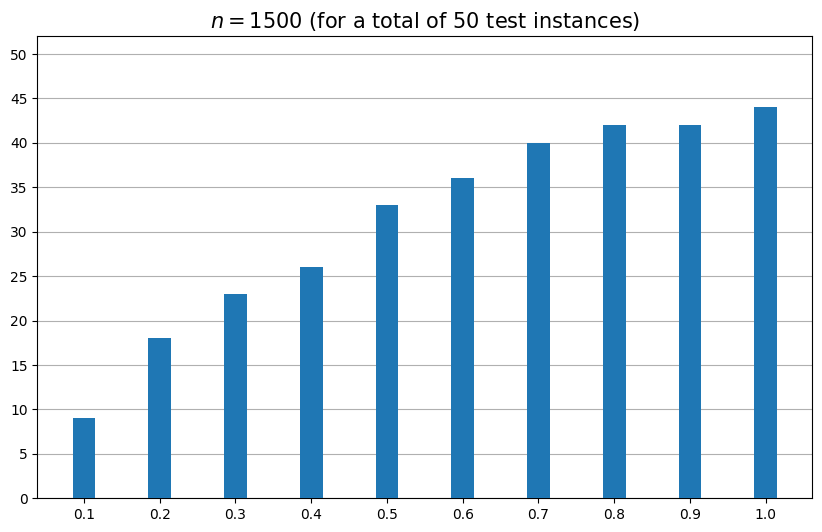}
    \end{subfigure}

    \begin{subfigure}
        \centering
        \includegraphics[width=0.49\textwidth]{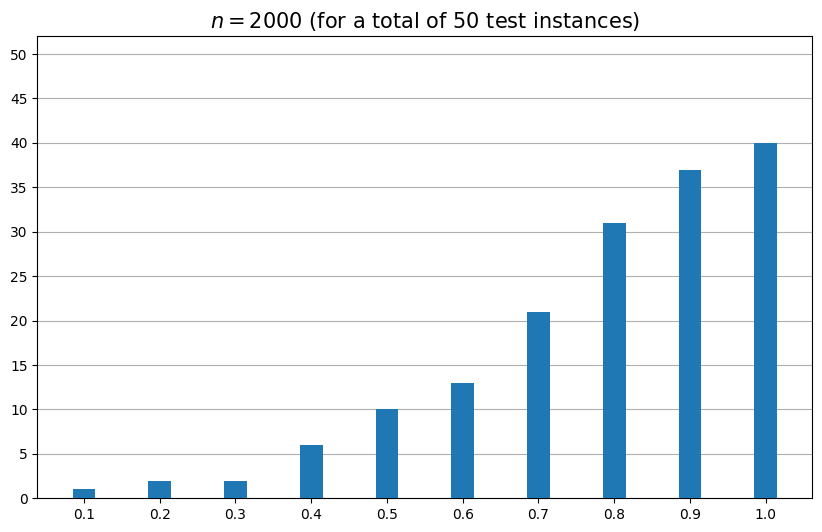}
    \end{subfigure}
    \begin{subfigure}
        \centering
        \includegraphics[width=0.49\textwidth]{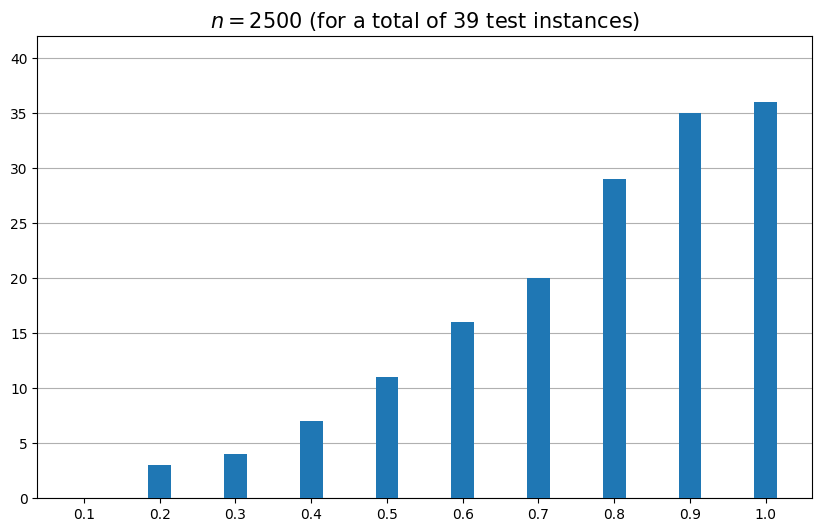}
    \end{subfigure}

    \caption{Number of \ref{scp} instances (y-axis) for each  time-reduction factor (x-axis)}\label{fig:numinst}
\end{figure}

\subsection{Limiting the number of problems {\texorpdfstring{\rm\ref{fj}}{Fj}} solved for strong fixing}\label{subsec:limitsf}
Solving problems \ref{fj} 
%(and similarly \ref{fj1}) 
for all $j\in\{1,\ldots,n\}$ can take a long time and may not compensate for the effort to fix variables by the strong-fixing procedure. Here, we analyze the trade-off between the number of problems \ref{fj} solved, and the number of variables fixed at 0 and the time to solve \ref{scp}. 
% We consider for this analysis,  the then instances of \ref{CP1} generated by Alg. \ref{alg:gen} that are described in Table \ref{tab:inslimk}, where  we see  that all the graphs $G$ generated  have 60 nodes and around 200 edges, and   the dimension of the constraint matrix  $A\in\mathbb{R}^{m\times n}$ of \ref{scp} is given by   $n=2000$ and $m\approx 30000$. We also see the time to solve the instances by  \texttt{Gurobi} and the number of variables that remains  not fixed after applying  reduced-cost fixing ($n_{\mbox{\tiny{RC}}}$) and  strong fixing ($n_{\mbox{\tiny{SF}}}$).
We consider for this analysis,  50 instances of \ref{CP1} generated by Alg. \ref{alg:gen}, for which  the graphs $G$  have 60 nodes and around 200 edges, and   the dimension of the constraint matrix  $A\in\mathbb{R}^{m\times n}$ of \ref{scp} is given by   $n\!=\!2000$ and $m\!\approx\! 30000$. In Table \ref{tab:inslimk}, we show details for 10  test instances, specifically, the number of nodes and edges of $G$, the dimension of the constraint matrix  $A\in\mathbb{R}^{m\times n}$ of \ref{scp}, the time to solve the instances by  \texttt{Gurobi}, and the number of variables that were fixed at 0 when  applying  reduced-cost fixing ($n_{\mbox{\tiny{RC}}}^0$) and  strong fixing ($n_{\mbox{\tiny{SF}}}^0$).

\begin{table}[!ht]
\centering{
\scriptsize{
\begin{tabular}{c|ccccrcc}
\multicolumn{1}{c|}{\#}&$|\mathcal{V}(G)|$&$|\mathcal{I}(G)|$&$m$&$n$&time &$n_{\rm RC}^0$&$n_{\rm SF}^0$\\[2pt]
\hline\\[-6pt]
% 1  & 60 & 210 & 29227 & 2000 & 163.56 & 1793                     & 1221 \\
% 2  & 60 & 210 & 29206 & 2000 & 204.67 & 1828                     & 1318 \\
% 3  & 60 & 214 & 28927 & 2000 & 745.25 & 1862                     & 1557 \\
% 4  & 60 & 218 & 30746 & 2000 & 6895.17 & 1998                     & 1944 \\
% 5  & 60 & 214 & 31606 & 2000 & 1721.46 & 1997                     & 1897 \\
% 6  & 60 & 208 & 28552 & 2000 & 346.24 & 1917                     & 1621 \\
% 7  & 60 & 214 & 29785 & 2000 & 273.65 & 1927                     & 1524 \\
% 8  & 60 & 216 & 30911 & 2000 & 4232.38 & 1943                     & 1888 \\
% 9  & 60 & 216 & 29569 & 2000 & 601.11 & 1952                     & 1620 \\
% 10 & 60 & 222 & 31244 & 2000 & 2806.55 & 1992 & 1848\\
1  & 60 & 210 & 29227 & 2000 & 163.56  & 207 & 779 \\
2  & 60 & 210 & 29206 & 2000 & 204.67  & 172 & 682 \\
3  & 60 & 214 & 28927 & 2000 & 745.25  & 138 & 443 \\
4  & 60 & 218 & 30746 & 2000 & 6895.17 & 2   & 56  \\
5  & 60 & 214 & 31606 & 2000 & 1721.46 & 3   & 103 \\
6  & 60 & 208 & 28552 & 2000 & 346.24  & 83  & 379 \\
7  & 60 & 214 & 29785 & 2000 & 273.65  & 73  & 476 \\
8  & 60 & 216 & 30911 & 2000 & 4232.38 & 57  & 112 \\
9  & 60 & 216 & 29569 & 2000 & 601.11  & 48  & 380 \\
10 & 60 & 222 & 31244 & 2000 & 2806.55 & 8   & 152
\end{tabular}
}
}
\caption{Impact of limiting the number of \ref{fj} solved for strong fixing}\label{tab:inslimk}
\end{table}

When limiting to an integer $k<n$, the number of problems \ref{fj} that we solve, i.e., solving them only for $j\in S\subset 
\{1,\ldots,n\}$, with $|S|=k$,  the selection of the subset $S$ is crucial. We will compare results for two different strategies of selection, which we denote in the following by `Jaccard' and `Best-$\mathfrak
    {z}_{j}$'. 
Specifically, after solving \ref{fj} for $j=i$, we solve it for $j=\hat\imath$   such that the variable $z_{\hat{\imath}}$ is still not fixed, and satisfies one of the following criteria. We denote by  $J_i\subset\{1,\ldots,n\}$, the set of indices associated to the variables that are still not fixed  after solving F$_i^{\mbox{\tiny SCP(0)}}$, and we denote the value of the variable $u$ at the optimal solution of  F$_i^{\mbox{\tiny SCP(0)}}$ by $u^*(\mbox{F$_i^{\mbox{\tiny SCP(0)}}$})$.       
     \begin{itemize}
    \item Jaccard:  $A_{\cdot \hat{{\imath}}}$ is the closest column to $A_{\cdot i}$\,, according to the `Jaccard similarity'. For  a pair of columns of $A$, this similarity is measured by the cardinality of the intersection of the supports of the columns  divided by the cardinality of the union of the supports. More specifically, 
    \[
     \hat{\imath}:=\mbox{argmax}_{j\in J_i}\{|\mbox{supp}(A_{\cdot j})\cap \mbox{supp}(A_{\cdot i})|/|\mbox{supp}(A_{\cdot j})\cup \mbox{supp}(A_{\cdot i})|\}.
    \]
    When using this strategy, the  initial problem solved is the one corresponding to the point on the upper left corner of the rectangle where the points in $N$ are positioned.

    The goal of this strategy is to change the objective of the problem F$_j^{\mbox{\tiny SCP(0)}}$ with respect to the one that was solved immediately before, the least possible. As the feasible region of the problems are all the same, by doing this we expect to solve the problems more quickly by warm-starting  with the optimal solution of the preceding one. 
   
    \item Best-$\mathfrak
    {z}_{j}$: $z_{ \hat{\imath}}$ is the `first' variable not fixed, that is, it is the variable that is still not fixed and  for which the objective function value of \ref{fj} computed at $u^*(\mbox{F$_i^{\mbox{\tiny SCP(0)}}$})$ is closest to the upper bound UB (notice that the objective function value of \ref{fj} is smaller than UB for all non-fixed variables). More specifically, we set
\[
  \hat{\imath}:=\mbox{argmax}_{j\in J_i}\{w_j+(u^*(\mbox{F$_i^{\mbox{\tiny SCP(0)}}$}))^\top(\mathbf{e} -  A_{\cdot j})-\mbox{UB}\}.
\]
When using this strategy, the initial problem solved is the one corresponding to the smallest right-hand side of \eqref{ineq-fixing} for the variables still not fixed after solving  the linear relaxation of \ref{scp} and applying  reduced-cost fixing. 
\end{itemize}

To evaluate both these selection strategies of the set $S$, we  also obtain the best possible selection for the set, following the two steps below.

\begin{enumerate}
       \item We solve problems \ref{fj}, for all $j\in \{1,\ldots,n\}$, and save $S_j$, which we define as the set of indices corresponding to the variables that we can fix with the solution of \ref{fj}.
       \item Then, we solve the  problem, where for a given integer $k$, with $1\leq k\leq n$, we select the $k$ problems \ref{fj} for which we can fix the maximum number of variables. Next, we will formulate this problem.

       We are given the sets $S_{j}$, for $j\in \{1,\ldots,n\}$, and the integer $k$; and we construct the  $n\times n$ matrix $F$, such that $f_{ij}=1$ if $i\in S_{j}$, and $f_{ij}=0$ otherwise. We  define the variable $x_j\in\{0,1\}$ for each $j\in \{1,\ldots,n\}$, which is 1 if we solve \ref{fj}, and 0 otherwise;  and the variable $y_j\in\{0,1\}$, for each $j\in \{1,\ldots,n\}$, which is 1 if the variable $z_j$ could be fixed by our solution, and 0 otherwise. Then, we formulate the problem as  
\begin{equation}\label{prob:best-k}\tag{Best-$k$}
\begin{array}{l}
    \max~ \sum_{j\in N} y_j\\
    Fx - y \geq 0,\\
     \sum_{j\in N} x_j\leq k,\\
    x_j,y_j\in \{0,1\},~ \forall~j\in \{1,\ldots,n\}.
\end{array}
\end{equation}
The optimal objective value of \ref{prob:best-k} gives the maximum number of variables that  can be fixed by solving $k$ problems \ref{fj}.
\end{enumerate}

In Figs. \ref{fig:limitk} and \ref{fig:limitk2}, we present 
  results from our experiments.  
 On the horizontal axis of the plots we have $k$ as the fraction of the total number of problems \ref{fj} that are solved, i.e., for each $k=0.1, 0.2,\ldots, 1$, we solve $kn$ problems. We present plots for the two selection strategies that we propose (`Jaccard' and `Best-$\mathfrak{z}_{j}$').  For comparison purposes, we also consider the best possible selection determined by the solution of \ref{prob:best-k} (`Best-$k$'). 
    
In Fig. \ref{fig:limitk}, we present average results for the 50 instances tested.     In its first sub-figure, 
we have  fractions of the total number of variables that can be fixed at 0 by solving all the $n$ problems \ref{fj}.  We see that `Best-$\mathfrak{z}_{j}$'  is more effective than `Jaccard' in fixing variables with the solution of few problems \ref{fj}, approaching the maximum number of variables that can be fixed, given by `Best-$k$'.
   In the second sub-figure,  
    we show the average time-reduction factor for the  instances, for the two selection strategies. 
    % The time-reduction factor is the ratio of the elapsed time used by \texttt{Gurobi} to solve the problem after we applied the matrix reduction and the strong fixing to the elapsed time used by \texttt{Gurobi} to solve the original problem. 
    % When the factor is less than one, it means that our presolve led to a problem that could be solved faster than the original problem. 
       
    In Fig. \ref{fig:limitk2}, we show  the \emph{total-}time reduction factor for the  strategy `Best-$\mathfrak{z}_{j}$', for each instance described in Table \ref{tab:inslimk},  where the word `total' represents the fact that we now count for the time spent solving problems \ref{fj}, i.e., the time spent on our pre-solving is added to the time to solve \ref{scp}. 
    
    %Analyzing Figs. \ref{fig:limitk2} and \ref{fig:limitk} together, w
    We see that when using the `Best-$\mathfrak{z}_{j}$' selection strategy, by solving $0.4n$ problems \ref{fj}, we can reduce the solution time  for 5 out of 10 instances analyzed. The reduction in time considering the presolve time can be obtained because on average more than 80\% of the number of variables possibly fixed by strong fixing can already be fixed when solving only 40\% of the problems \ref{fj} (see Fig. \ref{fig:limitk}).

    We can conclude that limiting the number of problems \ref{fj} solved can lead to a practical use of the strong-fixing procedure. 

\begin{figure}[!ht]
\begin{center}
\includegraphics[width=0.7\textwidth]{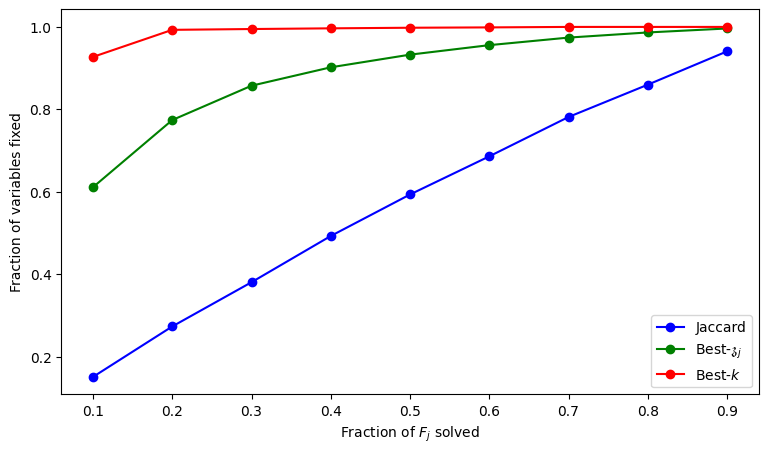}
\includegraphics[width=0.7\textwidth]{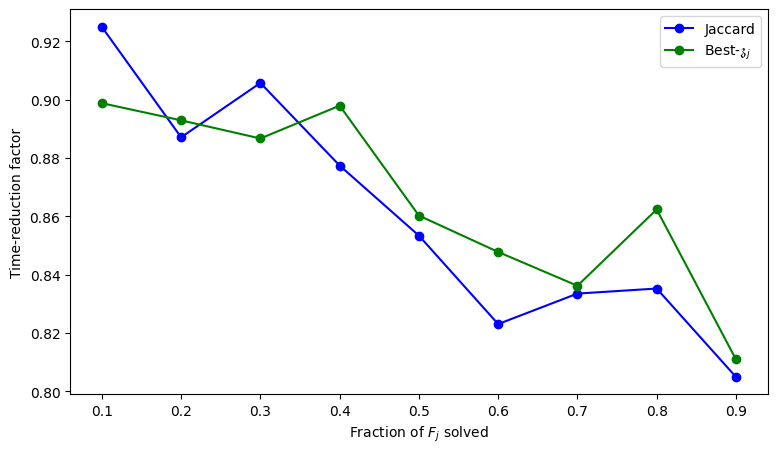}
%{time_MIP.png}
\caption{Average impact of number of problems 
 \ref{fj} solved for strong fixing}\label{fig:limitk}
\end{center}
\end{figure}

\begin{figure}[!ht]
\begin{center}
\includegraphics[width=0.7\textwidth]{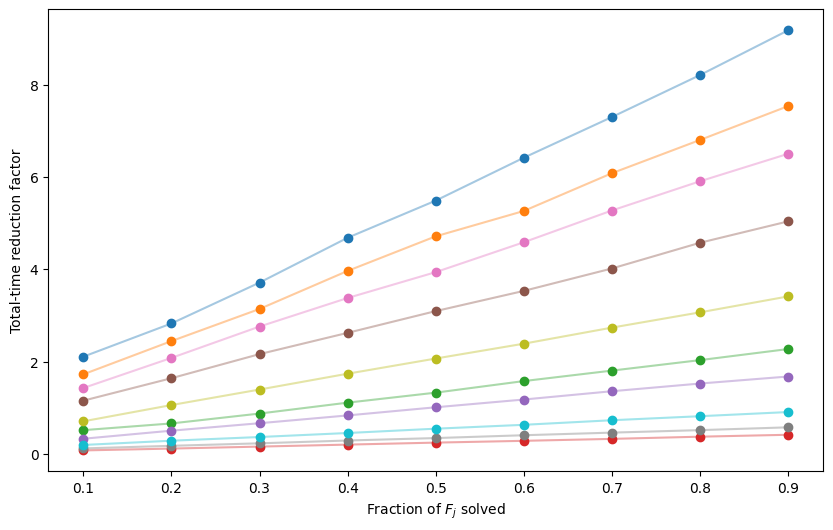}
%{total-time.png}%
\caption{Impact of number of problems 
 \ref{fj} solved on the  total solution  time with strong fixing for the strategy `Best-$\mathfrak
    {z}_{j}$', for each instance of Table \ref{tab:inslimk}}\label{fig:limitk2}
\end{center}
\end{figure}

\section{Selecting the center of the covering balls from a convex set: the \ref{CP2} problem}\label{sec:cp2}

We now have a version of our problem where  the exact position of the center of each ball is not prescribed. Rather, for each ball, we are given a bounded convex set where the center could be located. 
We will now assume that 
 if a ball covers
\emph{some} point in an edge, then it covers the entire edge.
This assumption is not extremely restrictive, because  we can break each edge into smaller edges, 
limiting the modeling error, albeit at the expense of  adding more nodes to the graph. 

We note that in the previous version of the problem we restrict the set where the center of each ball could be located to an enumerable set of points. Again,  we can always increase the number of points, 
limiting the modeling error, albeit at the expense of increasing the size of the set-covering problem to be solved. 

The more appropriate model will depend on which approximation is  better for the case studied, the one given by the discretization of the set where the center of each ball could be located or by the partition of the edges into smaller edges such that each one should be completely covered by a unique ball.  

Similar to $G$ before, $H$ is a straight-line embedding of a graph in $\mathbb{R}^d$, where we denote  the vertex set of $H$ by $\mathcal{V}(H)$,
and the edge set of $H$ by $\mathcal{I}(H)$, which is a finite set of closed intervals.

Before, we were given a finite set $N$ of $n$ points in $\mathbb{R}^d$, which were the 
possible locations to center a ball at. 
Now, instead, we define $N:=\{1,2,\ldots,n\}$
to be the index set of possible locations 
for locating a ball.
As before, we have a weight function $w:N\rightarrow \mathbb{R}_{++}$\,, and  covering radii $r:N\rightarrow \mathbb{R}_{++}$\,.

We let $Q_j\subset \mathbb{R}^d$ be the allowable locations associated with $j\in N$. We will assume that $Q_j$ is ``tractable'', in the sense that
we can impose membership of
a point in $Q_j$ as a constraint within a
tractable convex optimization model. 

We define the variables
\begin{itemize}
\item $x_j\in\mathbb{R}^d$, the center of a ball indexed by $j\in N$.
\item For $j\in N$, the indicator variable $z_j=1$ if our solution uses the ball indexed by  $j$ in the solution (so its cost is considered in the objective value), and $z_j=0$, otherwise. 
\item The indicator variable $y_{ij}=1$ if our solution uses the ball indexed by $j$ to covers edge $i$, and $y_{ij}=0$, otherwise.
\end{itemize}

We only need a  $y_{ij}$  variable 
if the following $d$-variable convex system has a feasible solution:
\[
x\in Q_j\,,~
\|x-a_{i(1)}\|\leq r_j\,,~
\|x-a_{i(2)}\|\leq r_j\,,
\mbox{ for edge } i:=[a_{i(1)}\,,a_{i(2)}].
\]

Then, for each edge $i\in \mathcal{I}(H)$, we let $N_i\subset N$ denote the set of $j$ for which the system above has a feasible solution, and we use $N_i$ in the model to define the constraints. 

We are now prepared to formulate our problem as the following mixed integer second order cone program (MISOCP):
\begin{equation}\tag{CP2}\label{CP2}
    \begin{array}{lll}
    \min&\!\! \sum_{j\in N}  w_j z_j\\
    \mbox{s.t.}&\!\!z_j\geq y_{ij}\,,&\!\! \forall i \!\in \!\mathcal{I}(H),\, j\!\in\! N_i\,;\\
   &\!\! \sum_{j\in N_i} 
   y_{ij} \geq 1, &\!\!\forall i \!\in\! \mathcal{I}(H);\\
   &\!\!\|x_j-a_{i(k)}\|\!\leq \!r_j \!  +\!M_{ijk}(1 - y_{ij}),
   & \!\! \forall  i\!\in\! \mathcal{I}(H),\, j\!\in \!N_i\,,\,  k\!\in\!\{1,2\};\\
   &\!\!y_{ij}\in \{0,1\},&\!\! \forall i\!\in\! \mathcal{I}(H),\,  j\!\in\! N_i\,;\\
   &\!\!z_j\in\mathbb{R}, ~ x_j\in Q_j\subset \mathbb{R}^d, &\!\!\forall j\! \in\! N.
\end{array}\end{equation}

Our model has  Big-M coefficients.
For each edge $i\in \mathcal{I}(H)$,  ball index $j\in N_i$, and $k=1,2$, we define 
$M_{ijk} :=D_{ijk}-r_j$ where $D_{ikj}$ is  the maximum distance between  any point in $Q_j$ and the  endpoint  $a_{i(k)}$\,,  with the edge $i:=[a_{i(1)}\,,a_{i(2)}]$. 
To compute $D_{ijk}$\,,
assuming that $Q_j$ is a polytope,
we compute the Euclidean distance between each
vertex of $Q_j$ and the endpoint $a_{i(k)}$ of the edge $i$.
$D_{ijk}$ is then the maximum distance computed (over all the vertices of $Q_j$)\,.
Finally, we note that although it is not explicitly imposed in the model, we  have $z_j\in\{0,1\}$, for all $j\in N$, at an optimal solution of \ref{CP2}.

%%%%%%%%%%%%%%%%%%

\section{Strong fixing for \ref{CP2}}\label{sec:strongfixCP2}

Aiming at applying the strong fixing methodology discussed in \S\ref{sec:strongfix} to \ref{CP2}, we   now consider its continuous  relaxation, i.e., the following  second-order cone program (SOCP),  

\begin{equation}\tag{$\overline{\mbox{CP2}}$}\label{barCP2}
    \begin{array}{lll}
    \min&\!\! \sum_{j\in N}  w_j z_j\\
     
    \mbox{s.t.}&\!\!z_j\geq y_{ij}\,,& \!\!\forall i \!\in \!\mathcal{I}(H),\, j\!\in\! N_i\,;\\
   &\!\! \sum_{j\in N_i} 
   y_{ij} \geq 1, &\!\!\forall i \!\in\! \mathcal{I}(H);\\
   &\!\!\|x_j-a_{i(k)}\|\!\leq \!r_j \!  +\!M_{ijk}(1 - y_{ij}),
   &  \!\!\forall  i\!\in\! \mathcal{I}(H),\, j\!\in \!N_i\,,\,  k\!\in\!\{1,2\};\\
   &\!\!y_{ij}\geq 0,& \!\! \forall i\!\in\! \mathcal{I}(H),\,  j\!\in\! N_i\,;\\
   &\!\!\! z_j\in\mathbb{R}, ~ x_j\in Q_j\subset \mathbb{R}^d, &\!\!\forall j\! \in\! N,
\end{array}\end{equation}
and the associated dual problem. To simplify the presentation, we will assume here that  $Q_j$ is a polytope. For all  $j \in N$, let $Q_{1j}\,,Q_{2j}$ be given vectors in $\mathbb{R}^d$, and define the polytope $Q_j$ by
$
Q_j:=\{x\in\mathbb{R}^d~:~ Q_{2j}\leq x \leq Q_{1j}\}.
$ In this case, defining the sets $\mathcal{I}^j\subset \mathcal{I}(H)$, for all  $j\in N$, as 
$
\mathcal{I}^j:=\{i\in\mathcal{I}(H) ~:~j\in N_i\},
$ we can formulate the Lagrangian dual problem of  \ref{barCP2}  as (see  
%a detailed analysis in 
\ref{app:dualCP2}) 
\begin{equation}\tag{$\mbox{D2}$}\label{DRCP2}
\begin{array}{lll}
    \!\!\!\max &\multicolumn{2}{l}{\!\displaystyle
    \sum_{j\in N} (\theta_{2j}^\top Q_{2j}-\theta_{1j}^\top Q_{1j}) + \!\sum_{i \in \mathcal{I}(H)} (\mu_i + \sum_{j \in N_i}\sum_{k=1}^2 (\gamma_{ijk}^\top  a_{i(k)} - \nu_{ijk} D_{ijk})  ) }  \\
    \text{s.t.}&\! w_j- \displaystyle\sum_{i \in \mathcal{I}^j} \lambda_{ij} = 0,&        \!\!\!\!\forall j \in N;\\
     &\!\lambda_{ij}-\mu_i + \!
     \displaystyle\sum_{k=1}^2\nu_{ijk} M_{ijk}\! -\!\beta_{ij}=  0,&   \!\!\!\forall  i\in \mathcal{I}(H),\,   j \in N_i\,;\\
    & \!\theta_{1j} -\theta_{2j} -\displaystyle\sum_{i \in \mathcal{I}^j} \displaystyle\sum_{k=1}^2\gamma_{ijk}=  0,&    \!\!\!\!\forall   j \in N;\\
    &\!||\gamma_{ijk}||\leq \nu_{ijk}\,,&    \!\!\!\!\forall    i\in \mathcal{I}(H),\, j \in N_i\,,\,k\in \{1,2\};\\
    &\!\lambda_{ij},\mu_i,\nu_{ijk},\beta_{ij},\gamma_{ijk} \geq 0,&    \!\!\!\!\forall   i\in \mathcal{I}(H),\, j \in N_i\,,\,k\in \{1,2\};\\
    &\!\theta_{1j},\theta_{2j} \geq 0,&    \!\!\!\!\forall j \in N;\\
    &\!\lambda_{ij},\mu_i,\nu_{ijk},\beta_{ij} \in \mathbb{R},&    \!\!\!\!\forall   i\in \mathcal{I}(H),\, j \in N_i\,,k\in \{1,2\};\\
    &\!\gamma_{ijk} \in \mathbb{R}^d,&      \!\!\!\!\forall i\in \mathcal{I}(H),\,j \in N_i\,,  k\in \{1,2\};\\
    &\!\theta_{1j},\theta_{2j} \in \mathbb{R}^d,&  \!\!\!\!\forall   j \in N.
    \end{array}
    \end{equation}

Based on convex duality, in Thm. \ref{thm:fix_dopt2}  we address the generalization  of the technique  of \emph{reduced-cost fixing}, discussed in \S\ref{sec:strongfix}, to \ref{CP2}. The result is well known and has been widely used in the literature (for example, see \cite{Ryoo1996ABA}). For completeness, we present its proof for our particular case in \ref{app:fixCP2}.

\begin{theorem}\label{thm:fix_dopt2}
Let UB be the objective-function value of a feasible solution for  {\rm\ref{CP2}}, and let $(\hat{\lambda},\hat{\mu},\hat{\nu},\hat{\beta},\hat{\gamma},\hat{\theta}_1,\hat{\theta}_2)$ be a feasible solution for  {\rm\ref{DRCP2}} with objective value $\hat\xi$.
Then, for every optimal solution $(z^\star,y^\star,x^\star)$ for {\rm\ref{CP2}}, we have:
%\vspace{-3pt}
\begin{align}
     &y_{ij}^\star = 0,\quad \forall i\in\mathcal{I}(H), j \in N_i ~ \text{such that} ~ \hat{\beta}_{ij} > UB -\hat\xi.\label{ineq-fixing2}
     \end{align}
 \end{theorem}
 
We observe that any feasible solution to \ref{DRCP2} can be used in \eqref{ineq-fixing2}. Then, for all $(i,j)$, with  $i\in\mathcal{I}(H)$ and $j \in N_i$\,, we propose the solution of 
\begin{align*}
\mathfrak{z}_{ij}^{\mbox{\tiny CP2}}:=&\max \beta_{ij}\! + \!\!\sum_{j\in N} (\theta_{2j}^\top Q_{2j}-\theta_{1j}^\top Q_{1j}) 
+ \!\!\sum_{i \in \mathcal{I}(H)}\! (\mu_i + \!\!\sum_{j \in N_i}\sum_{k=1}^2 (\gamma_{ijk}^\top a_{i(k)} - \nu_{ijk} D_{ijk}) )\\
&\quad \mbox{ s.t. } ({\lambda},{\mu},{\nu},{\beta},{\gamma},{\theta}_1,{\theta}_2) \mbox{ be a feasible solution to  \ref{DRCP2}.} \label{fj2}\tag{F$_{ij}^{\mbox{\tiny CP2}}$}
\end{align*}
For each pair $(i,j)$ such that $i\in\mathcal{I}(H),~ j \in N_i$\,, if there is a feasible solution $(\hat{\lambda},\hat{\mu},\hat{\nu},\hat{\beta},\hat{\gamma},\hat{\theta}_1,\hat{\theta}_2)$ to \ref{DRCP2} that can be used in \eqref{ineq-fixing2} to fix $y_{ij}$ at 0, then  the optimal solution of \ref{fj2} has objective value greater than $UB$ and can be used as well. 
Thus, we can consider solving all problems \ref{fj2} to fix the maximum number of variables $y_{ij}$ in \ref{CP2} at 0. However, from the following result, we see that we can still achieve the same goal by solving a smaller number of problems.

\begin{theorem}\label{thm:fix_z}
    Suppose that we are able to conclude from Thm. \ref{thm:fix_dopt2} that  for every optimal solution $(z^\star,y^\star,x^\star)$ for {\rm\ref{CP2}}, we have $y_{\hat{\imath}\hat{\jmath}}^*=0$, for a given pair $(\hat{\imath},\hat{\jmath})$, with $\hat{\imath}\in \mathcal{I}(H)$, $\hat{\jmath}\in N_{\hat{\imath}}$\,. Then, we may also conclude that $z_{\hat{\jmath}}^*=0$ and
    $y_{i\hat{\jmath}}^* = 0$, for all $i\in \mathcal{I}^{\hat{\jmath}}$,   for every optimal solution $(z^\star,y^\star,x^\star)$ for {\rm\ref{CP2}}.  
\end{theorem}

\begin{proof}
    If $y_{\hat{\imath}\hat{\jmath}}^*=0$ at every optimal solution for \ref{CP2}, we cannot have $z_{\hat{\jmath}}=1$ at any optimal solution, otherwise, setting $y_{\hat{\imath}\hat{\jmath}}=1$ at this optimal solution, we would not change its objective value and therefore, we would still have an optimal solution for \ref{CP2}, contradicting the hypothesis. Then, as $z_{\hat{\jmath}}^*=0$, by the first constraints in \ref{CP2}, we see that   $y_{i\hat{\jmath}}^*=0$ for all  $i\in \mathcal{I}^{\hat{\jmath}}$.
\end{proof}

We note that we can view the fixing of $z_j$ variables motivated by Thm. \ref{thm:fix_z},
as the same as strong fixing on  $z_j$ variables with respect to a model that 
includes the redundant constraint $z\geq 0$ in
\ref{barCP2}. Following the approach discussed in \S\ref{sec:strongfix} for \ref{scp}, we can  also include the redundant constraint $z\leq \mathbf{e}$ in \ref{barCP2}, aiming at fixing  variables at 1.   

To further explain our strong fixing on $z_j$, let \ref{RCP2zj} be the continuous relaxation of \ref{CP2} with the additional redundant constraints $z\geq 0$ and $z\leq \mathbf{e}$. Let \ref{DRCP2zjapp} be the Lagrangian dual of \ref{RCP2zj}, and let $\phi$ and $\delta$ be the additional dual variables to those in \ref{DRCP2}, respectively associated with the redundant constraints. Problems \ref{RCP2zj} and \ref{DRCP2zjapp} are  presented in \ref{app:dualCP2zj}, and the proof of the following result (similar to Thm. \ref{thm:fix_dopt2}) is in \ref{app:fixCP2zj}.  

\begin{theorem}\label{thm:fix_dopt2zj}
Let UB be the objective-function value of a feasible solution for  {\rm\ref{CP2}}, and let $(\hat{\lambda},\hat{\mu},\hat{\nu},\hat{\beta},\hat{\gamma},\hat{\theta}_1,\hat{\theta}_2,\hat{\phi},\hat{\delta})$ be a feasible solution for  {\rm\ref{DRCP2zjapp}} with objective value $\hat\xi$.
Then, for every optimal solution $(z^\star,y^\star,x^\star)$ for {\rm\ref{CP2}}, we have:
\vspace{-5pt}
\begin{align}
     &z_{j}^\star = 0,\quad \forall j \in N ~ \text{such that} ~ \hat{\phi}_{j} > UB -\hat\xi,\label{ineq-fixing2zj0}\\
          &z_{j}^\star = 1,\quad \forall  j \in N ~ \text{such that} ~ \hat{\delta}_{j} > UB -\hat\xi.\label{ineq-fixing2zj1}
     \end{align}
 \end{theorem}

Based on Thm. \ref{thm:fix_dopt2zj},  we propose the solution of the problems
\begin{align*}\label{fj2zj0}\tag{F$_{j}^{\mbox{\tiny CP2(0)}}$}
\mathfrak{z}_{j}^{\mbox{\tiny CP2(0)}}:=&\max \phi_{j}\! + \!\!\textstyle\sum_{j\in N} (\theta_{2j}^\top Q_{2j}-\theta_{1j}^\top Q_{1j} -\delta_j) \\ 
&\qquad \quad + \!\textstyle\sum_{i \in \mathcal{I}(H)} (\mu_i + \sum_{j \in N_i}\sum_{k=1}^2 (\gamma_{ijk}^\top a_{i(k)} - \nu_{ijk} D_{ijk})  )\\
&\mbox{ s.t. } ({\lambda},{\mu},{\nu},{\beta},{\gamma},{\theta}_1,{\theta}_2,\phi,\delta) \mbox{ be a feasible solution to  {\rm\ref{DRCP2zjapp}},} 
\end{align*}
and
\begin{align*}\label{fj2zj1}\tag{F$_{j}^{\mbox{\tiny CP2(1)}}$}
\mathfrak{z}_{j}^{\mbox{\tiny CP2(1)}}:=&\max \delta_{j}\! + \!\!\textstyle\sum_{j\in N} (\theta_{2j}^\top Q_{2j}-\theta_{1j}^\top Q_{1j} -\delta_j) \\
&\qquad \quad + \!\textstyle\sum_{i \in \mathcal{I}(H)} (\mu_i + \sum_{j \in N_i}\sum_{k=1}^2 (\gamma_{ijk}^\top a_{i(k)} - \nu_{ijk} D_{ijk})  )\\
&\quad \mbox{ s.t. } ({\lambda},{\mu},{\nu},{\beta},{\gamma},{\theta}_1,{\theta}_2,\phi,\delta) \mbox{ be a feasible solution to  {\rm\ref{DRCP2zjapp}},} 
\end{align*}
for all $j\in N$.
If the value of the optimal solution  of \ref{fj2zj0} (resp., \ref{fj2zj1})   is greater than $UB$, we can fix $z_j$ at 0 (resp., at 1).

Our \emph{strong fixing} procedure for \ref{CP2} fixes all possible variables $z_j$ in \ref{CP2} at 0 and 1, 
in the context of Thm. \ref{thm:fix_dopt2zj},
by using  a given upper bound $UB$ on the optimal solution value of \ref{CP2}, and solving   problems \ref{fj2zj0} and \ref{fj2zj1},  for all $j\in N$.

%%%%%%%%%%%%%%%%%%%%%%%%%%%%%%%%%

\section{Computational experiments for \ref{CP2}}\label{sec:compCP2}

Our numerical experiments in this section have two goals: to analyze the efficiency of the strong-fixing procedure for \ref{CP2}, and  to compare the solutions of \ref{CP2} and \ref{scp}.

We generated test instances in the same way we used to test  \ref{scp}, that is, with  Alg. \ref{alg:gen}. 
The edges in $H$ are given by all the subintervals computed in step 10 of Alg. \ref{alg:gen}, and the convex sets $Q_j$, $j\!=\!1,\ldots,n$, are boxes centered on the $n$ points generated in step 5, with sides of length $\ell\!=\!0.05$. We note that if we had $\ell\!=\!0$, \ref{CP2} would be equivalent to \ref{scp}; by increasing $\ell$, we give more flexibility to the possible locations of the centers of the circles that must cover $H$.
Because solving the MISOCP \ref{CP2} is more expensive than solving \ref{scp}, and as it gives more flexibility in selecting SLS locations, we experiment here with smaller instances and with $\nu\!=\!0.05n$. Thus, compared to the experiments in \S\ref{sec:comp}, where we set $\nu=0.03n$, we have fewer SLS candidates for a given graph with $\nu$ nodes, but we can locate them at any point in the given boxes.  We generated 10 instances for each   $n\!\in\!\{20,40,60,80,100\}$.  Our framework is the same described in \S\ref{sec:comp}. Our implementation is in \texttt{Python}, using  \texttt{Gurobi}  v. 10.0.2 to solve  the optimization problems. 

%The time limit to solve each instance of \ref{CP2} was set to  2 hours.  

% Addressing our first goal of analyzing the impact of strong fixing, we show in  Table \ref{tab:cp2_sf}, the following statistics for each  instance: 
%          \begin{enumerate}
%              \item the optimal objective value of \ref{scp} and  \ref{CP2}, 
%              \item the number of SLSs used in the solutions of \ref{scp} and  \ref{CP2},
%              %\item the elapsed time to solve \ref{CP2},  
%             % \item the return code of the solver (or say it reached the time limit),
%            %  \item the optimality gap (in case it reached the time limit), 
%              \item the number of variables fixed by Thm. \ref{thm:fix_dopt2} using the optimal solution of \ref{DRCP2},
%              \item the number of variables fixed by strong fixing.
%     \end{enumerate}

    In Table \ref{tab:cp2_sf}, we show statistics for our experiment. The number of nodes and  edges in  $H$, and the number of SLS  candidates, are given respectively by `$|\mathcal{V}(H)|$', `$|\mathcal{I}(H)|$', and `$n$'. The value of the optimal solution  for the two proposed formulations is given by `opt \ref{scp}' and `opt \ref{CP2}'. The number of SLSs installed in the optimal solutions for \ref{scp} and \ref{CP2} is given by `$\Sigma z_j^*$' for each problem. The elapsed times to solve \ref{CP2} before and after we fix variables by the strong fixing procedure are given respectively by `\texttt{Gurobi} original time' and `\texttt{Gurobi} reduced time', and the number of variables fixed at 0 and 1 by strong fixing are given respectively  by `$n_{\rm SF}^0$' and `$n_{\rm SF}^1$'. Finally, the time to run the strong fixing procedure is given by `strong fixing time'.  
    
    Comparing the optimal solution values of both models, we see that the flexibility regarding the location of SLSs is largely exploited by the solution for \ref{CP2}, which reduces the cost of SLS installations by 24\% on average.
It is interesting to note that this cost reduction is not always achieved by installing fewer SLSs. In fact, for some cases, the cost reduction was achieved by installing 3 or 4 more SLSs. 
We conclude that although \ref{CP2} is computationally harder to solve than \ref{scp}, if there are large areas where an SLS can be constructed at any point in each, it may be worthwhile to use this MISOCP formulation to optimize the chosen locations. 

    Regarding  variable fixing, we  note that no variable could be fixed by the standard reduced-cost fixing strategy, where the optimal values of the dual variables $\beta_{ij}$, corresponding to the constraints $y_{ij}\geq 0$ in \ref{barCP2},  are used in Thm. \ref{thm:fix_dopt2}. On the other hand, applying the strong fixing strategy was effective for several instances. We first applied strong fixing on the variables $y_{ij}$  by solving \ref{fj2}. Then, motivated by Thms. \ref{thm:fix_z} and \ref{thm:fix_dopt2zj}, we applied strong fixing on the variables $z_{j}$  by solving \ref{fj2zj0} and \ref{fj2zj1}, aiming to fix them at 0 or 1. We observed that the same variables $z_j$ that could be fixed at 0, by solving   \ref{fj2}, for each variable $y_{ij}$, and considering Thm. \ref{thm:fix_z}, could be fixed by solving \ref{fj2zj0} for each variable $z_j$. The clear advantage of this observation is that, in general, there are many fewer problems \ref{fj2zj0} to be solved than problems \ref{fj2}.
    Furthermore, when solving \ref{fj2zj1} we could also fix  variables $z_j$ at 1. 
 %     It is interesting to note that including redundant constraints in \ref{barCP2} does not help to fix more variables when using the optimal dual variables associated with them in
 % Thm. \ref{thm:fix_dopt2}, but from the results in Table \ref{tab:cp2_sf} we see that for strong fixing the use of redundant constraints was effective for several instances.
    Six out of the ten smallest instances were solved to optimality solely through strong fixing, having all variables fixed at 0 or 1.

\begin{table}%[!ht]
\tabcolsep=5pt
\centering{
\scriptsize{
\begin{tabular}{r|rrrrrrrrrrrr}
&&&&\multicolumn{1}{c}{}&\multicolumn{1}{c}{}&&&\multicolumn{1}{c}{\texttt{Gurobi}}&&&\multicolumn{1}{c}{strong}&\multicolumn{1}{c}{\texttt{Gurobi}}\\
&&&&\multicolumn{1}{c}{opt}&\multicolumn{1}{c}{opt}&\multicolumn{1}{c}{$\Sigma z_j^*$}&\multicolumn{1}{c}{$\Sigma z_j^*$}&\multicolumn{1}{c}{original}&&\multicolumn{1}{c}{}&\multicolumn{1}{c}{fixing}&\multicolumn{1}{c}{reduced}\\
\multicolumn{1}{c|}{\#}&$|\mathcal{V}(H)|$&$|\mathcal{I}(H)|$&$n$&\ref{scp} &\ref{CP2}&\ref{scp} &\ref{CP2}&\multicolumn{1}{c}{time}&\multicolumn{1}{c}{$n_{\rm SF}^0$}&\multicolumn{1}{c}{$n_{\rm SF}^1$}&\multicolumn{1}{c}{time}&\multicolumn{1}{c}{time}\\[2pt]
\hline\\[-6pt]
1  & 10 & 161  & 20  & 0.45 & 0.25 & 4  & 3  & 6.58    & 15 & 1 & 5.80    & 1.78    \\
2  & 10 & 135  & 20  & 0.37 & 0.30 & 7  & 7  & 3.46    & 13 & 7 & 4.15    & 0.00    \\
3  & 10 & 164  & 20  & 0.38 & 0.26 & 4  & 5  & 5.75    & 14 & 3 & 5.91    & 1.84    \\
4  & 10 & 126  & 20  & 0.37 & 0.24 & 6  & 4  & 3.64    & 16 & 4 & 3.87    & 0.00    \\
5  & 10 & 189  & 20  & 0.46 & 0.35 & 3  & 2  & 8.32    & 18 & 2 & 7.15    & 0.00    \\
6  & 10 & 194  & 20  & 0.62 & 0.53 & 3  & 3  & 9.21    & 17 & 3 & 8.72    & 0.00    \\
7  & 10 & 137  & 20  & 0.27 & 0.20 & 5  & 4  & 4.46    & 16 & 4 & 4.50    & 0.00    \\
8  & 10 & 307  & 20  & 0.16 & 0.16 & 2  & 2  & 12.23   & 18 & 2 & 11.16   & 0.00    \\
9  & 10 & 165  & 20  & 0.33 & 0.27 & 3  & 3  & 7.20    & 16 & 2 & 6.37    & 1.55    \\
10 & 10 & 156  & 20  & 0.61 & 0.37 & 8  & 7  & 4.12    & 4  & 2 & 6.51    & 3.12    \\
11 & 20 & 421  & 40  & 0.62 & 0.42 & 10 & 6  & 29.34   & 14 & 1 & 33.19   & 11.71   \\
12 & 20 & 441  & 40  & 0.54 & 0.41 & 10 & 7  & 36.81   & 16 & 2 & 36.41   & 14.13   \\
13 & 20 & 466  & 40  & 0.25 & 0.19 & 6  & 3  & 66.61   & 21 & 0 & 38.69   & 13.05   \\
14 & 20 & 419  & 40  & 0.63 & 0.44 & 10 & 9  & 48.33   & 10 & 0 & 38.45   & 29.78   \\
15 & 20 & 477  & 40  & 0.37 & 0.34 & 8  & 7  & 64.52   & 5  & 1 & 42.96   & 41.14   \\
16 & 20 & 498  & 40  & 0.34 & 0.27 & 8  & 7  & 83.13   & 1  & 0 & 59.09   & 144.98   \\
17 & 20 & 536  & 40  & 0.58 & 0.43 & 8  & 8  & 60.18   & 10 & 0 & 58.58   & 24.64   \\
18 & 20 & 430  & 40  & 0.58 & 0.45 & 6  & 5  & 26.83   & 15 & 0 & 39.36   & 14.77   \\
19 & 20 & 367  & 40  & 0.37 & 0.27 & 8  & 7  & 17.09   & 12 & 1 & 30.32   & 11.75   \\
20 & 20 & 412  & 40  & 0.36 & 0.26 & 10 & 9  & 59.90   & 10 & 0 & 41.48   & 32.20   \\
21 & 30 & 791  & 60  & 0.43 & 0.32 & 7  & 10 & 229.94  & 6  & 0 & 153.03  & 175.96  \\
22 & 30 & 730  & 60  & 0.46 & 0.34 & 9  & 11 & 199.66  & 0  & 0 & 121.65  & 199.66  \\
23 & 30 & 825  & 60  & 0.45 & 0.37 & 11 & 10 & 225.36  & 0  & 0 & 155.59  & 225.36  \\
24 & 30 & 828  & 60  & 0.26 & 0.20 & 11 & 7  & 93.40   & 1  & 0 & 150.03  & 175.27   \\
25 & 30 & 827  & 60  & 0.28 & 0.20 & 7  & 8  & 197.69  & 9  & 0 & 155.22  & 150.33  \\
26 & 30 & 826  & 60  & 0.28 & 0.21 & 9  & 8  & 528.42  & 5  & 0 & 150.14  & 244.13  \\
27 & 30 & 920  & 60  & 0.63 & 0.46 & 11 & 9  & 336.27  & 0  & 0 & 181.68  & 336.27  \\
28 & 30 & 779  & 60  & 0.44 & 0.34 & 12 & 12 & 128.48  & 1  & 0 & 129.27  & 205.00  \\
29 & 30 & 832  & 60  & 0.35 & 0.27 & 10 & 8  & 178.64  & 1  & 0 & 152.29  & 176.81  \\
30 & 30 & 820  & 60  & 0.33 & 0.28 & 8  & 8  & 172.71  & 0  & 0 & 143.70  & 172.71  \\
31 & 40 & 1195 & 80  & 0.42 & 0.29 & 10 & 12 & 802.56  & 0  & 0 & 314.79  & 802.56  \\
32 & 40 & 1228 & 80  & 0.30 & 0.23 & 10 & 13 & 627.10  & 1  & 0 & 332.78  & 403.99  \\
33 & 40 & 1199 & 80  & 0.64 & 0.46 & 13 & 10 & 636.39  & 0  & 0 & 346.73  & 636.39  \\
34 & 40 & 1150 & 80  & 0.55 & 0.39 & 11 & 13 & 678.80  & 2  & 0 & 294.97  & 846.46  \\
35 & 40 & 1392 & 80  & 0.27 & 0.22 & 11 & 12 & 1186.35 & 0  & 0 & 454.88  & 1186.35 \\
36 & 40 & 1297 & 80  & 0.62 & 0.46 & 13 & 10 & 4861.33 & 0  & 0 & 362.33  & 4861.33 \\
37 & 40 & 1150 & 80  & 0.48 & 0.38 & 10 & 9  & 685.53  & 2  & 0 & 302.88  & 805.20  \\
38 & 40 & 1455 & 80  & 0.50 & 0.38 & 12 & 11 & 926.16  & 3  & 0 & 467.83  & 827.88  \\
39 & 40 & 1257 & 80  & 0.65 & 0.50 & 8  & 9  & 610.34  & 0  & 0 & 362.82  & 610.34  \\
40 & 40 & 1221 & 80  & 0.25 & 0.18 & 10 & 12 & 1484.83 & 0  & 0 & 344.11  & 1484.83 \\
41 & 50 & 1713 & 100 & 0.42 & 0.34 & 10 & 10 & 7368.55 & 0  & 0 & 743.31  & 7368.54 \\
42 & 50 & 1788 & 100 & 0.29 & 0.24 & 9  & 9  & 7282.17 & 0  & 0 & 791.29  & 7282.17 \\
43 & 50 & 2101 & 100 & 0.42 & 0.35 & 9  & 8  & 3981.00 & 0  & 0 & 1059.55 & 3981.00 \\
44 & 50 & 1695 & 100 & 0.47 & 0.39 & 11 & 9  & 7364.34 & 0  & 0 & 665.12  & 7364.34 \\
45 & 50 & 1856 & 100 & 0.30 & 0.26 & 12 & 8  & 7397.04 & 0  & 0 & 798.32  & 7397.04 \\
46 & 50 & 1866 & 100 & 0.49 & 0.34 & 9  & 8  & 5948.14 & 3  & 0 & 854.47  & 2877.52 \\
47 & 50 & 1811 & 100 & 0.35 & 0.29 & 10 & 9  & 7376.13 & 0  & 0 & 732.93  & 7376.13 \\
48 & 50 & 1914 & 100 & 0.46 & 0.40 & 12 & 10 & 7402.76 & 0  & 0 & 845.46  & 7402.76 \\
49 & 50 & 1850 & 100 & 0.38 & 0.30 & 8  & 12 & 5292.99 & 0  & 0 & 815.92  & 5292.99 \\
50 & 50 & 1744 & 100 & 0.60 & 0.48 & 11 & 8  & 7376.88 & 0  & 0 & 722.58  & 7376.88
\end{tabular}
}}
\caption{\ref{CP2} reduction experiment and solution comparison with \ref{scp}}
\label{tab:cp2_sf}
\end{table}

In Table \ref{tab:cp2mean}, we show the shifted geometric mean with shift parameter 1, for the same statistics presented in Table \ref{tab:cp2_sf}, for 10 instances of each $n=20,40,60,80,100$.

\begin{table}%[!ht]
\tabcolsep=5pt
\centering{
\scriptsize{
\begin{tabular}{rrrrrrrrrrrr}
&&&\multicolumn{1}{c}{}&\multicolumn{1}{c}{}&&&\multicolumn{1}{c}{\texttt{Gurobi}}&&&\multicolumn{1}{c}{strong}&\multicolumn{1}{c}{\texttt{Gurobi}}\\
&&&\multicolumn{1}{c}{opt}&\multicolumn{1}{c}{opt}&\multicolumn{1}{c}{$\Sigma z_j^*$}&\multicolumn{1}{c}{$\Sigma z_j^*$}&\multicolumn{1}{c}{original}&&\multicolumn{1}{c}{}&\multicolumn{1}{c}{fixing}&\multicolumn{1}{c}{reduced}\\
$|\mathcal{V}(H)|$&$|\mathcal{I}(H)|$&$n$&\ref{scp} &\ref{CP2}&\ref{scp} &\ref{CP2}&\multicolumn{1}{c}{time}&\multicolumn{1}{c}{$n_{\rm SF}^0$}&\multicolumn{1}{c}{$n_{\rm SF}^1$}&\multicolumn{1}{c}{time}&\multicolumn{1}{c}{time}\\[2pt]
\hline\\[-6pt] 
        10 & 167.95 & 20 & 0.40 & 0.29 & 4.20 & 3.72 & 6.06 & 13.89 & 2.73 & 6.14 & 0.56  \\ 
        20 & 444.36 & 40 & 0.46 & 0.34 & 8.27 & 6.57 & 44.68 & 9.64 & 0.37 & 40.96 & 23.79  \\ 
        30 & 816.53 & 60 & 0.39 & 0.30 & 9.36 & 8.99 & 205.64 & 1.25 & 0.00 & 148.48 & 200.94  \\ 
        40 & 1250.92 & 80 & 0.46 & 0.34 & 10.71 & 11.01 & 968.02 & 0.53 & 0.00 & 354.39 & 951.71  \\ 
        50 & 1830.54 & 100 & 0.42 & 0.34 & 10.02 & 9.03 & 6559.97 & 0.15 & 0.00 & 796.92 & 6100.54 
\end{tabular}
}}
\caption{\ref{CP2} reduction experiment and solution comparison with \ref{scp} (shifted geometric mean for 10 instances of each size)}
\label{tab:cp2mean}
\end{table}
    
    Although a more careful implementation of strong fixing is needed to make it more practical, as concluded in \S\ref{sec:comp} for \ref{scp}, we see it as a promising procedure to reduce the sizes of the instances of \ref{CP2} and lead to their faster solution.
Comparing the solution times of \ref{CP2} before and after we fixed the variables, on the 31 instances where variables were fixed, we have an average time-reduction factor of 60\%. The time-reduction factor is the ratio of the elapsed time used by \texttt{Gurobi} to solve the problem after we applied strong fixing, to the elapsed time used by \texttt{Gurobi} to solve the original problem. We would like to highlight the impact of strong fixing on instance 46, for which fixing 3 variables at 0 led to a time-reduction factor of 48\% ($2877.52/5948.14$), and even if we consider the time spent on the strong fixing procedure we obtain a time-reduction factor of 62\%($(2877.52+854.47)/5948.14$).

We note that for \ref{CP2}, we do not analyze strategies to limit the number of problems \ref{fj2zj0} and \ref{fj2zj1} solved, as we did for \ref{scp} in \S\ref{subsec:limitsf}. The reason is that the time to apply the strong-fixing procedure is not as significant compared to the time to solve \ref{CP2} as it is compared to the time to solve \ref{scp}, especially considering that for the instances solved for \ref{CP2}, the value of $n$ is much smaller.

%\FloatBarrier

%%%%%%%%%%%%%%%%%%%%%%%%%%%%%%%%%

\section{Outlook}\label{sec:conc}

%\subsection{\ref{scp}}
\noindent \emph{Extensions for \ref{scp}.} Our integer-programming approach can be extended to other geometric settings. We can use other metrics, or replace  balls $B(x,r(x))$ with arbitrary convex sets $B(x)$ for which we can compute the intersection of with each edge $I\in \mathcal{I}(G)$. In this way, $N$ is just an index set, and we
only need a pair of line searches, to determine the
endpoints of the intersection of 
$B(x)$ with each $I\in \mathcal{I}(G)$. 
We still have $|\mathcal{C}(I)|\leq 1\!+\!2n$, for each 
$I\in \mathcal{I}(G)$, and so the 
number of covering constraints in \ref{scp} is
at most $(1\!+\!2n)|\mathcal{I}(G)|$. 
Finally, we could take $G$ to be geodesically embedded on a sphere and the balls replaced by geodesic balls. 

\smallskip 
\noindent {\emph{Solvable cases for \ref{scp}.}}
We may seek to generalize the 
 algorithm mentioned in \S\ref{sec:fork} to arbitrary families of subtrees of unit-grid-graph trees (i.e., eliminate the fork-free condition but restrict to unit-grid graphs). 

 \smallskip 
\noindent {\emph{Extending our computational work on strong fixing.}}
An important direction is to reduce the time for strong fixing, so as to get a large 
number of variables fixed without solving 
all of the fixing subproblems
Additionally, our strong-fixing methodology is very general and could work well for other classes of mixed-integer optimization problems. In particular, based on our results, the method looks promising to try for any
mixed-integer minimization formulation with convex relaxation for which the 
relaxation bound is close to an upper bound given by some feasible solution of the formulation; in such a situation, reduced-cost fixing is likely to already have
some effectiveness, and strong fixing is likely to improve the situation.

\smallskip 
\noindent {\emph{Scaling to larger instances.}}
While we can solve rather large instances of  \ref{scp}, 
it remains a challenge to solve comparably large instances of \ref{CP2}.

\smallskip
\noindent {\emph{Application.}} The development and deployment of eVTOLs is at this moment still on the horizon. Our motivation for the development of  optimization methods that we have developed  continue to develop, is so that the optimization community can be at-the-ready for helping UAM systems become a cost effective and safe reality.
We believe that we have good evidence now that integer programming can play a key role. The applied value of our progress so far and further challenges will mainly be known as the technology is deployed. 
Time will tell. 

% \mf{future work:
% \begin{enumerate}
%     \item Try to solve the robust version of the problem, where each point has to be covered by 2 (or more) SLSs.
%      \item Use instances with the configuration from previous papers (Pelegrin et al. or Portoleau and D., cf. email August 28)
%      \item Read Barany Edmonds Wolsey
%      \item Think about decomposition of the graph, link to dynamic programming?
%      \item Extend the formulation to the case where the trajectory has to be designed at the same time as the location of the SLSs
%      \end{enumerate}
%      }

   \section*{Acknowledgments}
This work was  supported in part by  NSF grant DMS-1929284 while 
C. D'Ambrosio, M. Fampa and J. Lee were in residence at 
the Institute for Computational and Experimental Research in Mathematics (ICERM)
at Providence, RI, during the Discrete Optimization 
semester 
program, 2023.
C. D'Ambrosio was supported by the Chair ``Integrated Urban Mobility'', backed by L’X - \'Ecole Polytechnique and La Fondation de l’\'Ecole Polytechnique (The Partners of the Chair 
accept no liability related to this publication, for which the chair holder is solely liable).
M. Fampa was supported  by  CNPq grant 307167/2022-4.
J. Lee was supported by 
the Gaspard Monge Visiting Professor Program (\'Ecole Polytechnique), and from 
 ONR grants N00014-21-1-2135 and N00014-24-1-2694. F. Sinnecker was supported on a master's scholarship from CNPq. 
  J. Lee and M. Fampa acknowledge (i) helpful 
 conversations at ICERM with Zhongzhu Chen, on  variable fixing, and (ii) some helpful information 
 from Tobias Achterberg on \texttt{Gurobi} presolve. 
 The authors acknowledge the Centro Brasileiro de Pesquisas Físicas (CBPF/MCTI) for the computational support, particularly the access to high-performance systems used in the calculations performed in this work. This support was essential for the development and completion of the analyses presented.

% \subsubsection{\discintname}
%  The authors have no competing interests to declare.
%  %that are relevant to the content of this article.
% \end{credits}

%\subsubsection{Acknowledgments}  
\bibliographystyle{splncs04}
\bibliography{covering}

\appendix

\section{Proof of Thm. \ref{thm:fix_dopt}}\label{app:fixSCP0}
To prove Thm. \ref{thm:fix_dopt},  we consider a modified version of \ref{scp} where we add to it  the  constraint 
\begin{equation}\label{absurdineq}
     z_{\hat\jmath} \geq  \textstyle \left\lfloor
  \frac{UB-\hat{u}^\top\mathbf{e}}{w_{\hat\jmath} - \hat{u}^\top A_{\cdot {\hat\jmath}}}\right\rfloor +1,
\end{equation}
for some  ${\hat\jmath} \in \{1,\ldots,n\}$ such that  $w_{\hat\jmath} - \hat{u}^\top A_{\cdot {\hat\jmath}}>0$. In this case, the only difference in the dual of the continuous relaxation of the modified problem with respect to \ref{d} is the addition of the new dual variable $\sigma \in\mathbb{R}$ corresponding to this added constraint to the dual constraint corresponding to  the variable $z_{\hat\jmath}$\,, which becomes 
\[
\textstyle u^\top A_{\cdot j} + \sigma \leq w_j\,,
\]
and the addition of the term corresponding to $\sigma$ to the objective function, which becomes 
\[
\textstyle u^\top \mathbf{e} + \left(  \left\lfloor
  \frac{UB-\hat{u}^\top\mathbf{e}}{w_j - \hat{u}^\top A_{\cdot j}}\right\rfloor +1\right) \sigma.
\]
We note that $(\hat u, \hat\sigma:=w_{\hat\jmath} - \hat{u}^\top A_{\cdot \hat\jmath} )$ is feasible to the modified dual problem with objective value 
\[
\textstyle \hat{u}^\top \mathbf{e} + \left(  \left\lfloor
  \frac{UB-\hat{u}^\top\mathbf{e}}{w_j - \hat{u}^\top A_{\cdot j}}\right\rfloor +1\right) \left(w_{\hat\jmath} - \hat{u}^\top A_{\cdot {\hat\jmath}}\right),
\]
 which is a lower bound for the optimal value of the modified \ref{scp} that is strictly greater than the upper bound UB for the objective value of \ref{scp}. We conclude that no optimal solution to \ref{scp} can satisfy \eqref{absurdineq}.
\qed

\section{Proof of Thm. \ref{thm:fix_doptat1} }\label{app:fixSCP}

To prove Thm. \ref{thm:fix_doptat1},  we consider a modified version of \ref{scp} where we add to it  the  constraint $z_{\hat\jmath}=0$, for some  $\hat\jmath\in \{1,\ldots,n\}$\,. In this case, the only difference in the dual of the continuous relaxation of the modified problem with respect to \ref{dplus} is the addition of the new dual variable $\sigma \in\mathbb{R}$ corresponding to this added constraint to the dual constraint corresponding to  the variable $z_{\hat\jmath}$\,, which becomes 
\[
A_{\cdot j}^\top u - v_j + \sigma \leq w_j\,.
\]

We consider that $(\hat{u},\hat{v})$  
is feasible to \ref{dplus} with objective value $\hat{u}^\top \mathbf{e} - \hat{v}^\top\mathbf{e}$, and we define 
$\tilde{v}_{j}:=\hat{v}_{j} + \sigma$, if $j=\hat\jmath$, and 
$\tilde{v}_{j}:=\hat{v}_{j}$\,, otherwise.   Then $(\hat{u},\tilde{v},\sigma)$ is a feasible solution to the modified dual problem with objective value $\hat{u}^\top \mathbf{e} - \hat{v}^\top\mathbf{e} - \sigma$, if $\hat{v}_{\hat\jmath} + \sigma\geq 0$. To maximize the objective of the  modified dual, we take $\sigma =   -\hat{v}_{\hat\jmath}$\,, which gives a lower bound for the optimal value of the modified \ref{scp} equal to  $\hat{u}^\top \mathbf{e} - \hat{v}^\top\mathbf{e} + \hat{v}_{\hat\jmath}$\,. If this lower bound is strictly greater than a  given upper bound UB for the objective value of \ref{scp}, we conclude that no optimal solution to \ref{scp} can have $z_{\hat\jmath}=0$.
\qed

\section{Lagrangian duality for \ref{CP2}}\label{app:dualCP2}

We assume here that  $Q_j$ is a polytope. So, for given $Q_{1j}\,,Q_{2j}\in \mathbb{R}^d$, we have 
$
Q_j:=\{x\in\mathbb{R}^d~:~ Q_{2j}\leq x \leq Q_{1j}\},
$ 
 for all  $j \in N$.
 
We consider the continuous relaxation of  \ref{CP2}, given by 
\begin{equation}\tag{$\overline{\mbox{CP2}}$}\label{RCP2}
    \begin{array}{lll}
    \min&\!\! \sum_{j\in N}  w_j z_j\\
     
    \mbox{s.t.}&\!\!z_j\geq y_{ij},& \forall i \!\in \!\mathcal{I}(H), j\!\in\! N_i\,;\\
   &\!\! \sum_{j\in N_i} 
   y_{ij} \geq 1, &\forall i \!\in\! \mathcal{I}(H);\\
   &\!\!\|x_j-a_{i(k)}\|\!\leq \!r_j \!  +\!M_{ijk}(1 - y_{ij}),
   &  \forall  i\!\in\! \mathcal{I}(H), j\!\in \!N_i\,,  k\!\in\!\{1,2\};\\
 % & z_j\in \{0,1\}, \forall j\in N;\\
   &\!\!y_{ij}\geq 0,& \forall i\!\in\! \mathcal{I}(H),  j\!\in\! N_i\,;\\
   &\!\!Q_{2j}\leq x_j \leq Q_{1j}\,, &\forall j \!\in\! N;\\
   &\!\!z_j\in\mathbb{R}, ~ x_j\in  \mathbb{R}^d, &\forall j\! \in\! N,\\
   &\!\!y_{ij} \in \mathbb{R}, &  \forall~i\in \mathcal{I}(H),  \forall j\!\in \!N_i\,.
\end{array}\end{equation}
We introduce the variable $\rho_{ijk} \in \mathbb{R}^d$, defined as  
\begin{align*}
    % \|&\rho_{ijk}\|\leq r_j   +(D_{ijk}-r_j)(1 - y_{ij}),\\
   & \rho_{ijk} := x_j - a_{i(k)}\,,\qquad \forall  ~i:=[a_{i(1)},a_{i(2)}]\in \mathcal{I}(H), ~ j\in N_i\,, ~ k\in\{1,2\}.
   \end{align*}
Then, the Lagrangian function associated to \ref{RCP2} is:
\begin{align*}
&L(z, y, x, \rho ; \lambda, \mu, \nu, \gamma, \beta ,\theta ) :=\\ 
&\quad \sum_{j \in N} w_j z_j + \sum_{i \in \mathcal{I}(H)} \sum_{j \in N_i} \lambda_{ij} (y_{ij} - z_j)  + \sum_{i \in \mathcal{I}(H)} \mu_i \left(1 - \sum_{j \in N_i} y_{ij} \right) \\
&\quad + \sum_{i \in \mathcal{I}(H)} \sum_{j \in N_i} \sum_{k=1}^2 \nu_{ijk} \left( \|\rho_{ijk}\| - r_j - M_{ijk}(1 - y_{ij}) \right)\\
&\quad + \sum_{i \in \mathcal{I}(H)} \sum_{j \in N_i} \sum_{k=1}^2 \gamma_{ijk}^\top(\rho_{ijk}-x_j+a_{i(k)}) - \beta_{ij}y_{ij}\\
&\quad + \sum_{j\in N} \left(\theta_{1j}^\top(x_j-Q_{1j}) + \theta_{2j}^\top (Q_{2j}-x_j)\right),
\end{align*}
or equivalently,
\begin{align*}
&L(z, y, x, \rho ; \lambda, \mu, \nu, \gamma, \beta ,\theta ) := \\
&\qquad  \sum_{j \in N} w_jz_j -\sum_{i \in \mathcal{I}(H)}\sum_{j \in N_i}  \lambda_{ij} z_j \\
&\qquad + \sum_{i \in \mathcal{I}(H)} \sum_{j \in N_i} \left(\lambda_{ij}-\mu_i + \sum_{k=1}^2\nu_{ijk} M_{ijk} -\beta_{ij}\right) y_{ij}\\
&\qquad+ \sum_{i \in \mathcal{I}(H)}  \sum_{j \in N_i} \sum_{k=1}^2 v_{ijk} ||\rho_{ijk}||+ \gamma_{ijk}^\top \rho_{ijk}\\
&\qquad+\sum_{j\in N} \left(\theta_{1j} -\theta_{2j}\right)^\top x_j -\sum_{i \in \mathcal{I}(H)} \sum_{j \in N_i}\sum_{k=1}^2\gamma_{ijk}^{\top} x_j + \sum_{i \in \mathcal{I}(H)} \mu_i\\
&\qquad - \sum_{i \in \mathcal{I}(H)} \sum_{j \in N_i} \sum_{k=1}^2 \nu_{ijk} D_{ijk}+ \sum_{i \in \mathcal{I}(H)} \sum_{j \in N_i} \sum_{k=1}^2 \gamma_{ijk}^\top a_{i(k)} \\
% &\qquad - \sum_{j\in N} \alpha_{1j}\\
% &\qquad - \sum_{i \in \mathcal{I}(H)} \sum_{j \in N_i} \beta_{1ij}\\
&\qquad + \sum_{j\in N} \theta_{2j}^\top Q_{2j}-\theta_{1j}^\top Q_{1j}\,.
\end{align*}
We define  $\mathcal{I}^j\subset \mathcal{I}(H)$, for all $j\in N$, as 
$
\mathcal{I}^j:=\{i\in\mathcal{I}(H) ~:~j\in N_i\}.
$
To obtain the infimum of the Lagrangian function with respect to $(z, y, x, \rho )$, we note that 
\begin{align*}
    \inf_z &\sum_{j \in N} \left(w_j -\sum_{i \in \mathcal{I}^j}  \lambda_{ij}\right) z_j \\
    &=\begin{cases}
        0,& \mbox{if } w_j  -\sum_{i \in \mathcal{I}^j}  \lambda_{ij}= 0,\quad \forall j \in N;\\
        -\infty,&\mbox{otherwise,}
    \end{cases}\\
% \end{align*}
% \begin{align*}
    \inf_y&\sum_{i \in \mathcal{I}(H)} \sum_{j \in N_i} \left(\lambda_{ij}-\mu_i + \sum_{k=1}^2\nu_{ijk} M_{ijk}-\beta_{ij}\right) y_{ij} \\&=\begin{cases}
        0,&\mbox{if } \lambda_{ij}-\mu_i + \sum_{k=1}^2\nu_{ijk} M_{ijk}-\beta_{ij}=  0,\quad \forall  i\in \mathcal{I}(H),~j \in N_i\,;\\
        -\infty,& \mbox{otherwise,}
    \end{cases}\\
% \end{align*}
% \begin{align*}
    \inf_x & \sum_{j\in N} \left(\theta_{1j} -\theta_{2j} 
    -
    \sum_{i \in \mathcal{I}^j}  \sum_{k=1}^2\gamma_{ijk} \right)^\top x_j  
    \\
    &=\begin{cases}
        0,&\mbox{if } \theta_{1j} -\theta_{2j} -\left(\sum_{i \in \mathcal{I}^j} \sum_{k=1}^2\gamma_{ijk}\right) =  0,\quad \forall j \in N;\\
        -\infty,&\mbox{otherwise,}
    \end{cases}\\
    \inf_{\rho} &  \sum_{i \in \mathcal{I}(H)} \sum_{j \in N_i}\sum_{k=1}^2\nu_{ijk} ||\rho_{ijk}||+ \gamma_{ijk}^\top \rho_{ijk}    \\
    &=\begin{cases}
        0,&\mbox{if } ||\gamma_{ijk}||\leq \nu_{ijk}\,,\quad \forall  i\in \mathcal{I}(H),~j \in N_i\,,~k\in \{1,2\};\\
        -\infty, & \mbox{otherwise.}
    \end{cases}
\end{align*}
Then, the Lagrangian dual problem of \ref{RCP2} can then be formulated as 
\begin{equation}\tag{$\mbox{D2}$}\label{DRCP2app}
\begin{array}{lll}
    \!\!\max &\multicolumn{2}{l}{\!\displaystyle\sum_{j\in N} (\theta_{2j}^\top Q_{2j}-\theta_{1j}^\top Q_{1j}) + \!\!\sum_{i \in \mathcal{I}(H)} \mu_i + \sum_{j \in N_i}\sum_{k=1}^2 (\gamma_{ijk}^\top a_{i(k)} - \nu_{ijk} D_{ijk})  ) }  \\
    \!\text{s.t.}\!\!&\! w_j- \displaystyle\sum_{i \in \mathcal{I}^j} \lambda_{ij} = 0,&        \!\!\!\!\forall j \in N;\\
     &\!\lambda_{ij}\!-\!\mu_i \!+ \!\displaystyle\sum_{k=1}^2\nu_{ijk} M_{ijk}\! -\!\beta_{ij}=  0,&   \!\!\!\!\forall  i\in \mathcal{I}(H),   j \in N_i\,;\\
    & \!\theta_{1j} -\theta_{2j} -\displaystyle\sum_{i \in \mathcal{I}^j} \displaystyle\sum_{k=1}^2\gamma_{ijk}=  0,&    \!\!\!\!\forall   j \in N;\\
    &\!||\gamma_{ijk}||\leq \nu_{ijk}\,,&    \!\!\!\!\forall    i\in \mathcal{I}(H),j \in N_i\,,\,k\in \{1,2\};\\
    &\!\lambda_{ij},\mu_i,\nu_{ijk},\beta_{ij},\gamma_{ijk} \geq 0,&    \!\!\!\!\forall   i\in \mathcal{I}(H),j \in N_i\,,\,k\in \{1,2\};\\
    &\!\theta_{1j},\theta_{2j} \geq 0,&    \!\!\!\!\forall j \in N;\\
    &\!\lambda_{ij},\mu_i,\nu_{ijk},\beta_{ij} \in \mathbb{R},&    \!\!\!\!\forall   i\in \mathcal{I}(H),j \in N_i\,,\,k\in \{1,2\};\\
    &\!\gamma_{ijk} \in \mathbb{R}^d,&      \!\!\!\!\forall i\in \mathcal{I}(H),j \in N_i\,,\,  k\in \{1,2\};\\
    &\!\theta_{1j},\theta_{2j} \in \mathbb{R}^d,&  \!\!\!\!\forall   j \in N.
    \end{array}
    \end{equation}

\section{Proof of Thm. \ref{thm:fix_dopt2} }\label{app:fixCP2}

To prove Thm. \ref{thm:fix_dopt2}, we consider a modified version of \ref{CP2} where we add to it  the  constraint $y_{\hat\imath\hat\jmath}=1$, for some $\hat\imath\in\mathcal{I}(H)$ and $\hat\jmath\in N_i$\,. In this case, the only difference on the Lagrangian dual of the continuous relaxation of the modified problem with respect to \ref{DRCP2} is the subtraction of the new dual variable $\upsilon \in\mathbb{R}$ corresponding to this added constraint, from the objective function, and its addition to  the dual constraint corresponding to  the variable $y_{\hat\imath\hat\jmath}$\,, which becomes 
\[
\lambda_{\hat\imath\hat\jmath}-\mu_{\hat\imath} + \sum_{k=1}^2\nu_{\hat\imath\hat\jmath k} M_{\hat\imath\hat\jmath k} -\beta_{\hat\imath\hat\jmath} + \upsilon=  0.
\]

We consider that $(\hat{\lambda},\hat{\mu},\hat{\nu},\hat\beta,\hat{\gamma},\hat{\theta}_1,\hat{\theta}_2)$  is feasible to \ref{DRCP2} with objective value $\hat\xi$, and we define $\tilde\beta_{ij}:=\hat\beta_{ij} + \upsilon$, if $(i,j)=(\hat\imath,\hat\jmath)$, and $\tilde\beta_{ij}:=\hat\beta_{ij}$\,, otherwise.   Then $(\hat{\lambda},\hat{\mu},\hat{\nu},\tilde\beta,\hat{\gamma},\hat{\theta}_1,\hat{\theta}_2,\upsilon)$ is a feasible solution to the modified dual problem with objective value $\hat\xi -\upsilon$, if $\hat\beta_{\hat\imath\hat\jmath} + \upsilon\geq 0$. To maximize the objective of the  modified dual, we take $\upsilon = -  \hat\beta_{\hat\imath\hat\jmath}$\,, which gives a lower bound for the optimal value of the modified \ref{CP2} equal to  $\hat\xi + \hat\beta_{\hat\imath\hat\jmath}$\,. If this lower bound is strictly greater than a  given upper bound UB for the objective value of \ref{CP2}, we conclude that no optimal solution to \ref{CP2} can have $y_{\hat\imath\hat\jmath}=1$. \qed

\section{Lagrangian duality for \ref{CP2} with redundant constraints}\label{app:dualCP2zj}

We consider the continuous relaxation of  \ref{CP2} with the redundant constraints $0\leq z\leq\mathbf{e}$, given by 
\begin{equation}\tag{$\overline{\mbox{CP2}}^+$}\label{RCP2zj}
    \begin{array}{lll}
  \min&\!\! \sum_{j\in N}  w_j z_j\\
     
    \mbox{s.t.}&\!\!z_j\geq y_{ij}\,,& \forall i \!\in \!\mathcal{I}(H), j\!\in\! N_i\,;\\
   &\!\! \sum_{j\in N_i} 
   y_{ij} \geq 1, &\forall i \!\in\! \mathcal{I}(H);\\
   &\!\!\!\|x_j-a_{i(k)}\|\!\leq \!r_j \!  +\!M_{ijk}(1 \!- \!y_{ij}),
   &  \forall  i\!\in\! \mathcal{I}(H), j\!\in \!N_i\,,  k\!\in\!\{1,2\};\\
   &\!\!y_{ij}\geq 0,& \forall i\!\in\! \mathcal{I}(H),  j\!\in\! N_i\,;\\
   &\!\!Q_{2j}\leq x_j \leq Q_{1j}\,, &\forall j \!\in\! N;\\
   &\!\!0\leq z_j\leq 1
   , ~ x_j\in  \mathbb{R}^d, &\forall j\! \in\! N,\\
   &\!\!y_{ij} \in \mathbb{R}, &  \forall~i\in \mathcal{I}(H),  \forall j\!\in \!N_i\,.
\end{array}\end{equation}
With the same analysis of \ref{app:dualCP2} and considering $\phi,\delta\in\mathbb{R}^n$ as the Lagrangian multipliers associated to the new constraints $z\geq 0$ and $z\leq \mathbf{e}$, we formulate the 
the Lagrangian dual problem of \ref{RCP2zj} 
%can then be formulated 
as 
\begin{equation}\tag{$\mbox{D2$^+$}$}\label{DRCP2zjapp}
\begin{array}{lll}
    \!\!\!\!\max &\multicolumn{2}{l}{\!\!\!\!\!\displaystyle\sum_{j\in N} (\theta_{2j}^\top Q_{2j}\!-\!\theta_{1j}^\top Q_{1j}\!-\!\delta_j)\! +\! \!\!\sum_{i \in \mathcal{I}(H)} \!\!\!\mu_i\! +\!\! \sum_{j \in N_i}\sum_{k=1}^2 (\gamma_{ijk}^\top a_{i(k)} \!-\! \nu_{ijk} D_{ijk})  ) }  \\
    \text{s.t.}&\!\!\! w_j-\phi_j+\delta_j- \displaystyle\sum_{i \in \mathcal{I}^j} \lambda_{ij} = 0,&        \!\!\!\!\forall j \in N;\\
     &\!\!\!\lambda_{ij}\!-\!\mu_i + \!\displaystyle\sum_{k=1}^2\nu_{ijk} M_{ijk}\! -\!\beta_{ij}=  0,&   \!\!\!\!\forall  i\in \mathcal{I}(H),   j \in N_i\,;\\
    & \!\!\!\theta_{1j} -\theta_{2j} -\displaystyle\sum_{i \in \mathcal{I}^j} \displaystyle\sum_{k=1}^2\gamma_{ijk}=  0,&    \!\!\!\!\forall   j \in N;\\
    &\!\!\!||\gamma_{ijk}||\leq \nu_{ijk}\,,&    \!\!\!\!\forall    i\in \mathcal{I}(H),j \in N_i\,,\,k\in \{1,2\};\\
    &\!\!\!\lambda_{ij},\mu_i,\nu_{ijk},\beta_{ij},\gamma_{ijk} \geq 0,&    \!\!\!\!\forall   i\in \mathcal{I}(H),j \in N_i\,,\,k\in \{1,2\};\\
    &\!\!\!\theta_{1j},\theta_{2j},\phi_j,\delta_j \geq 0,&    \!\!\!\!\forall j \in N;\\
    &\!\!\!\lambda_{ij},\mu_i,\nu_{ijk},\beta_{ij} \in \mathbb{R},&    \!\!\!\!\forall   i\in \mathcal{I}(H),j \in N_i\,,\,k\in \{1,2\};\\
    &\!\!\!\gamma_{ijk} \in \mathbb{R}^d,&      \!\!\!\!\forall i\in \mathcal{I}(H),j \in N_i\,,\,  k\in \{1,2\};\\
    &\!\!\!\theta_{1j},\theta_{2j} \in \mathbb{R}^d,&  \!\!\!\!\forall   j \in N.
    \end{array}
    \end{equation}

\section{Proof of Thm. \ref{thm:fix_dopt2zj} }\label{app:fixCP2zj}

To prove Thm. \ref{thm:fix_dopt2zj}, we first consider a modified version of \ref{CP2} where we add to it  the  constraint $z_{\hat\jmath}=1$, for some  $\hat\jmath\in N_i$\,. In this case, the only difference on the Lagrangian dual of the continuous relaxation of the modified problem with respect to \ref{DRCP2zjapp} is the subtraction of the new dual variable $\upsilon \in\mathbb{R}$ corresponding to this added constraint, from the objective function, and its addition to the dual constraint corresponding to  the variable $z_{\hat\jmath}$, which becomes 
\[
w_{\hat\jmath}-\phi_{\hat\jmath}+\delta_{\hat\jmath}- \displaystyle\sum_{i \in \mathcal{I}^{\hat\jmath}} \lambda_{i{\hat\jmath}} +\upsilon= 0.
\]

We consider that $(\hat{\lambda},\hat{\mu},\hat{\nu},\hat\beta,\hat{\gamma},\hat{\theta}_1,\hat{\theta}_2,\hat{\phi},\hat{\delta})$  
is feasible to \ref{DRCP2zjapp} with objective value $\hat\xi$, and we define 
$\tilde{\phi}_{j}:=\hat{\phi}_{j} + \upsilon$, if $j=\hat\jmath$, and 
$\tilde{\phi}_{j}:=\hat{\phi}_{j}$\,, otherwise.   Then $(\hat{\lambda},\hat{\mu},\hat{\nu},\hat\beta,\hat{\gamma},\hat{\theta}_1,\hat{\theta}_2,\tilde{\phi},\hat{\delta},\upsilon)$ is a feasible solution to the modified dual problem with objective value $\hat\xi -\upsilon$, if $\hat{\phi}_{\hat\jmath} + \upsilon\geq 0$. To maximize the objective of the  modified dual, we take $\upsilon = -  \hat{\phi}_{\hat\jmath}$\,, which gives a lower bound for the optimal value of the modified \ref{CP2} equal to  $\hat\xi + \hat{\phi}_{\hat\jmath}$\,. If this lower bound is strictly greater than a  given upper bound UB for the objective value of \ref{CP2}, we conclude that no optimal solution to \ref{CP2} can have $z_{\hat\jmath}=1$.

Next, we consider a modified version of \ref{CP2} where we add to it  the  constraint $z_{\hat\jmath}=0$, for some  $\hat\jmath\in N_i$\,. In this case, the only difference on the Lagrangian dual of the continuous relaxation of the modified problem with respect to \ref{DRCP2zjapp} is the addition of the new dual variable $\upsilon \in\mathbb{R}$ corresponding to this added constraint to the dual constraint corresponding to the variable $z_{\hat\jmath}$, which becomes 
\[
w_{\hat\jmath}-\phi_{\hat\jmath}+\delta_{\hat\jmath}- \displaystyle\sum_{i \in \mathcal{I}^{\hat\jmath}} \lambda_{i{\hat\jmath}} +\upsilon= 0.
\]

We consider that $(\hat{\lambda},\hat{\mu},\hat{\nu},\hat\beta,\hat{\gamma},\hat{\theta}_1,\hat{\theta}_2,\hat{\phi},\hat{\delta})$  
is feasible to \ref{DRCP2zjapp} with objective value $\hat\xi$, and we define 
$\tilde{\delta}_{j}:=\hat{\delta}_{j} - \upsilon$, if $j=\hat\jmath$, and 
$\tilde{\delta}_{j}:=\hat{\delta}_{j}$\,, otherwise.   Then $(\hat{\lambda},\hat{\mu},\hat{\nu},\hat\beta,\hat{\gamma},\hat{\theta}_1,\hat{\theta}_2,\hat{\phi},\tilde{\delta},\upsilon)$ is a feasible solution to the modified dual problem with objective value $\hat\xi +\upsilon$, if $\hat{\delta}_{\hat\jmath} - \upsilon\geq 0$. To maximize the objective of the  modified dual, we take $\upsilon =   \hat{\delta}_{\hat\jmath}$\,, which gives a lower bound for the optimal value of the modified \ref{CP2} equal to  $\hat\xi + \hat{\delta}_{\hat\jmath}$\,. If this lower bound is strictly greater than a  given upper bound UB for the objective value of \ref{CP2}, we conclude that no optimal solution to \ref{CP2} can have $z_{\hat\jmath}=0$.
\qed

\end{document}